\documentclass[twoside,12pt]{article}
\usepackage{geometry}
\usepackage{amsmath,amssymb,amsfonts,amsthm}
\usepackage{txfonts}
\usepackage{todonotes}
\usepackage{tablists}
\usepackage{graphicx}
\usepackage{mathrsfs}
\usepackage{color}
\usepackage{float}
\usepackage{extarrows}
\usepackage{exscale}
\usepackage{titlesec}
\titleformat{\section}[block]{\large\center\sc}{\arabic{section}}{0.5em}{}[] 
\usepackage{relsize}
\usepackage{color}
\definecolor{teal}{RGB}{67,205,128}
\usepackage{hyperref}
\usepackage[titles]{tocloft} 
\usepackage{fancyhdr}
\theoremstyle{plain}
\newtheorem{theorem}{Theorem}[section]
\newtheorem{lemma}[theorem]{Lemma}
\newtheorem{corollary}[theorem]{Corollary}
\newtheorem{proposition}[theorem]{Proposition}

\newtheorem{definition}[theorem]{Definition}
\newtheorem{condition}[theorem]{Condition}
\newtheorem{problem}[theorem]{Problem}

\newtheorem{remark}[theorem]{Remark}

\let\oldsection\section
\renewcommand\section{\setcounter{equation}{0}\oldsection}

\def\be{\begin{equation}}
\def\ee{\end{equation}}
\def\bes{\begin{equation*}}
\def\ees{\end{equation*}}
\def\bs{\begin{split}}
\def\es{\end{split}}
\def\bali{\begin{aligned}}
\def\eali{\end{aligned}}
\newcommand{\pf}{\noindent {\bf Proof. \hspace{2mm}}}
\def\bR{{\mathbb R}}
\def\un{\underbrace}
\def\al{\alpha}
\def\e{\epsilon}
\def\ve{\varepsilon}

\def\t{\tilde}
\def\th{\theta}

\def\dl{\delta}
\def\Dl{\Delta}
\def\lt{\left}
\def\rt{\right}
\def\ls{\lesssim}

\def\i{\infty}
\def\p{\partial}
\def\f{\frac}
\def\na{\nabla}
\def\o{\omega}
\def\O{\Omega}

\def\q{\quad}

\def\bl{\boldsymbol}
\def\mS{\mathbb{S}}

\def\mH{\mathcal{H}}

\allowdisplaybreaks

\def\mS{\mathcal{S}}
\def\mfS{\mathfrak{S}}

\begin{document}
\title{\bf\normalsize  CONSTRAINED LARGE SOLUTIONS TO LERAY'S PROBLEM IN A DISTORTED STRIP WITH THE NAVIER-SLIP BOUNDARY CONDITION}

\author{\normalsize\sc Zijin Li, Xinghong Pan and Jiaqi Yang}

\date{}

\maketitle

\begin{abstract}
 In this paper, we solve the Leray's problem for the stationary Navier-Stokes system in a 2D infinite distorted strip with the Navier-slip boundary condition. The existence, uniqueness, regularity and asymptotic behavior of the solution are investigated. Moreover, we discuss how the friction coefficient affects the well-posedness of the solution. Due to the validity of the Korn's inequality, all constants in each a priori estimate are independent of the friction coefficient. Thus our method is also valid for the total-slip and no-slip cases. The main novelty is that the total flux of the velocity can be relatively large (proportional to the {\it slip length}) when the friction coefficient of the Navier-slip boundary condition is small, which is essentially different from the 3D case.

\medskip



\end{abstract}

\tableofcontents

\section{Introduction}\label{SEC1}
Consider the Navier-Stokes equations
\be\label{NS}
\lt\{
\begin{aligned}
&\bl{u}\cdot\na \bl{u}+\na p-\Dl \bl{u}=0,\\
&\na\cdot \bl{u}=0,
\end{aligned}
\rt.\q \text{in}\q \mathcal{S}\subset \bR^2,
\ee
subject to the Navier-slip boundary condition:
\be\label{NBC}
\left\{
\begin{aligned}
&2(\mathbb{S}\bl{u}\cdot\boldsymbol{n})_{\mathrm{tan}}+\alpha \bl{u}_{\mathrm{tan}}=0,\\
&\bl{u}\cdot\boldsymbol{n}=0,\\
\end{aligned}
\right.\q\text{on}\q\p\mathcal{S}.
\ee
Here $\mathbb{S}\bl{u}=\frac{1}{2}\left(\nabla \bl{u}+\nabla^T \bl{u}\right)$ is the stress tensor, where $\nabla^T \bl{u}$ represent the transpose of $\nabla \bl{u}$, and $\boldsymbol{n}$ is the unit outer normal vector of $\p \mS$. For a vector field $\bl{v}$, we denote $\bl{v}_{\mathrm{tan}}$ its tangential part:
\[
\bl{v}_{\mathrm{tan}}:=\bl{v}-(\bl{v}\cdot\boldsymbol{n})\boldsymbol{n},
\]
and $\al\geq0$ in \eqref{NBC} stands for the friction coefficient which may depend on various elements, such as the property of the boundary and the viscosity of the fluid. When $\al\to0_+$, the boundary condition \eqref{NBC} turns to be the total Navier-slip boundary condition, while when $\al\to\infty$, the boundary condition \eqref{NBC} degenerates into the no-slip boundary condition $u\equiv 0$. In this paper, we consider the general case, which assumes $0\leq\al\leq +\i$.

The domain  $\mS$ is a two dimensional infinitely distorted smooth strip as follows.
\begin{figure}[H]\label{FIG0}
\centering
\includegraphics[scale=0.3]{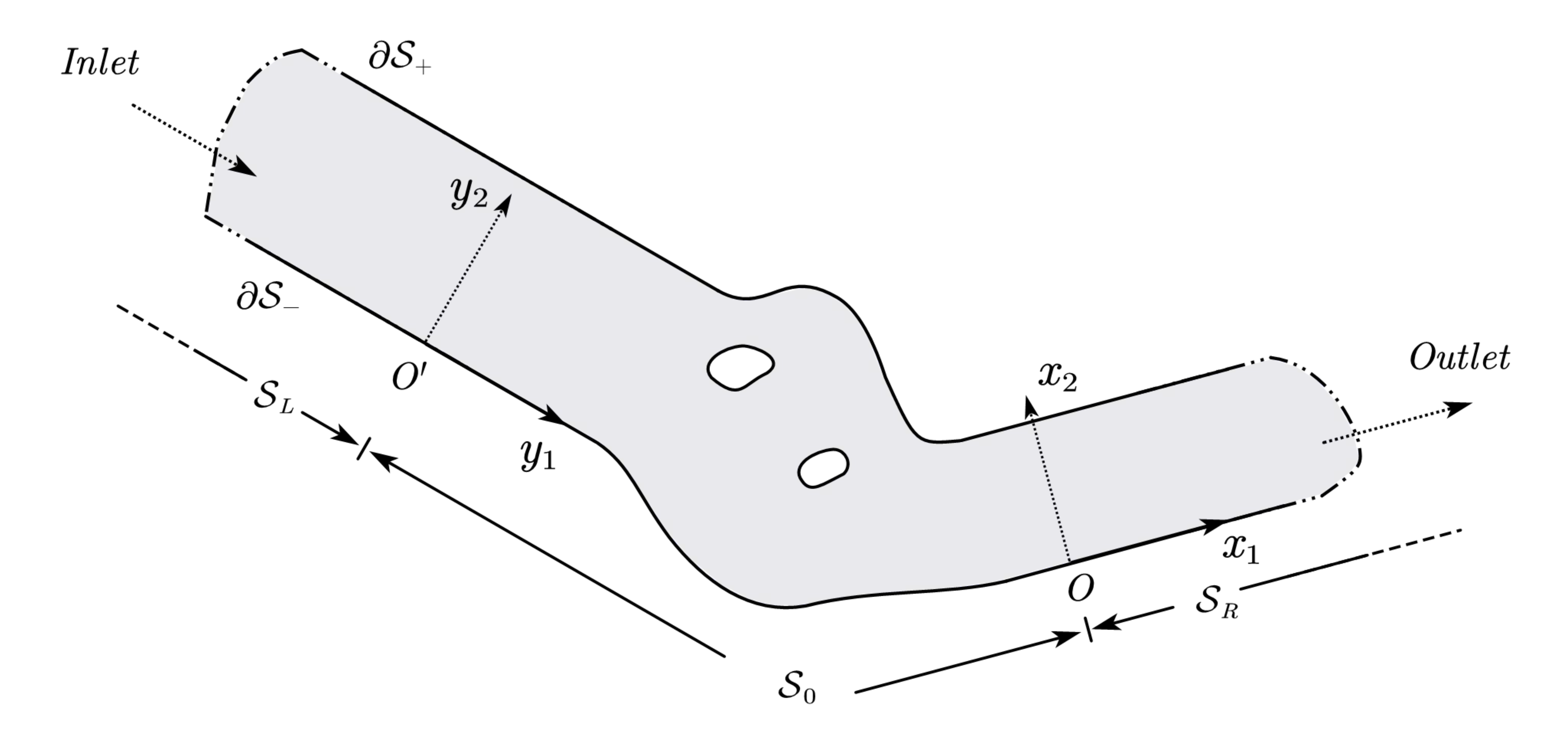}
\caption{The infinitely distorted strip $\mathcal{S}$.}
\end{figure}
Here, $\mathcal{S}_R$ and $\mathcal{S}_L$ are two semi-infinitely long straight strips. In the cartesian coordinate system $x_1Ox_2$, the strip
\[
\mathcal{S}_R:=\{x\in\mathbb{R}^2\,:\,(x_1,x_2)\in(0,\i)\times(0,1)\},
\]
while in the cartesian coordinate system $y_1O'y_2$, the stripe
\[
\mathcal{S}_L:=\{y\in\mathbb{R}^2\,:\,(y_1,y_2)\in(-\i,0)\times(0,c_0)\}.
\]
Here $c_0$ is a fixed constant. They are smoothly connected by the compact distorted part $\mathcal{S}_0$ in the middle. We do not insist the domain $\mathcal{S}$ to be simply connected, but all obstacles, with their boundaries are smooth Jordan curves, must lie in $\mathcal{S}_0$, and keep away from upper and lower boundaries of $\mathcal{S}$, i.e. $\p\mS_+$ and $\p\mS_-$, respectively.

Before stating our main results, we give some notations for later convenience.
\subsection*{Notations}

\q\ Throughout this paper, $C_{a,b,c,...}$ denotes a positive constant depending on $a,\,b,\, c,\,...$, which may be different from line to line. For simplicity, a constant $C_\mS$, depending on geometry properties of $\mS$, is usually abbreviated by $C$. Dependence on $\mS$ is default unless independence of $\mS$ is particularly stated. We use $\bl{e}_1,\,\bl{e}_2$ to denote the unit standard basis in the cartesian coordinate system $x_1Ox_2$, and $\bl{e}_1',\,\bl{e}_2'$ to denote the unit standard basis in the cartesian coordinate system $y_1O'y_2$. Meanwhile, for any $\zeta>1$, $\mathcal{S}_{R,\zeta}=(0,\zeta)\times(0,1)$ in the cartesian coordinate system $x_1Ox_2$, and $\mathcal{S}_{L,\zeta}=(-\zeta,0)\times(0,c_0)$ in the cartesian coordinate system $y_1O'y_2$. Then the truncated strip is defined by:
\be\label{Szeta}
\mS_\zeta:=\mathcal{S}_{L,\zeta}\cup\mathcal{S}_0\cup\mathcal{S}_{R,\zeta}.
\ee
We use $\Upsilon^{\pm}_\zeta$ to denote the right and left part of $\mathcal{S}_\zeta\backslash\mathcal{S}_{\zeta-1}$ as follows:
\[
\Upsilon^{+}_\zeta:=\mathcal{S}_{R,\zeta}\backslash\mathcal{S}_{R,\zeta-1},\q\q\Upsilon^{-}_\zeta:=\mathcal{S}_{L,\zeta}\backslash\mathcal{S}_{L,\zeta-1}.
\]
 We also apply $A\lesssim B$ to state $A\leq CB$. Moreover, $A\simeq B$ means both $A\lesssim B$ and $B\lesssim A$.

For $1\leq p\leq\infty$ and $k\in\mathbb{N}$, $L^p$ denotes the usual Lebesgue space with norm
\[
\|f\|_{L^p(D)}:=
\lt\{
\begin{aligned}
&\left(\int_{D}|f(x)|^pdx\right)^{1/p},\quad &1\leq p<\infty,\\[3mm]
&\mathrm{ess sup}_{x\in D}|f(x)|,\quad &p=\infty,\\
\end{aligned}
\rt.
\]
while $W^{k,p}$ denotes the usual Sobolev space with its norm
\[
\begin{split}
\|f\|_{W^{k,p}(D)}:=&\sum_{0\leq|L|\leq k}\|\nabla^L f\|_{L^p(D)},\\
\end{split}
\]
where $L=(l_1,l_2)$ is a multi-index. We also simply denote $W^{k,p}$ by $H^k$ provided $p=2$. Finally, $\overline{D}$ denote the closure of a domain $D$. A function $g\in W^{k,p}_{\mathrm{loc}}(D)$ or $W^{k,p}_{\mathrm{loc}}(\overline{D})$ function means $g\in W^{k,p}(\tilde{D})$, for any $\tilde{D}$ compactly contained in $D$ or $\overline{D}$.

For the 2D vector-valued function, we define
\[
\begin{split}
\mathcal{H}(\mS)&:=\left\{\bl{\varphi}\in H^1({\mS};\bR^2):\,\bl{\varphi}\cdot\bl{n}\big|_{\p \mS}=0 \right\},\\
\mathcal{H}_\sigma({\mS})&:=\left\{\bl{\varphi}\in H^1({\mS};\,\mathbb{R}^2):\,\na\cdot \bl{\varphi}=0,\ \bl{\varphi}\cdot \bl{n}\big|_{\p \mS}=0\right\},
\end{split}
\]
and
\[
\mathcal{H}_{\sigma,{\mathrm{loc}}}(\overline{\mS}):=\left\{\bl{\varphi}\in H^1_{{\mathrm{loc}}}(\overline{\mS};\,\mathbb{R}^2):\,\na\cdot \bl{\varphi}=0,\ \bl{\varphi}\cdot \bl{n}\big|_{\p \mS}=0\right\}.
\]
We also denote
\[
 \bl{X}:=\big\{\bl{\varphi}\in C_c^\infty(\overline{\mS}\,;\,\mathbb{R}^2):\ \na\cdot \bl{\varphi}=0,\ \bl{\varphi}\cdot\bl{n}\big|_{\p\mS}=0\big\}.
\]
Clearly, $\bl{X}$ is dense in $\mathcal{H}_\sigma$ in $H^1(\mS)$ norm. For matrices $\bl{\Gamma}=(\gamma_{ij})_{1\leq i,j\leq 2}$ and $\bl{K}=(\kappa_{ij})_{1\leq i,j\leq 2}$, we denote
\[
\bl{\Gamma}:\bl{K}=\sum^{2}_{i,j=1}\gamma_{ij}\kappa_{ij}.
\]

\subsection{The generalized Leray's problem in the distorted strip}

\q\ For a given flux $\Phi$ which is supposed to be nonnegative without loss of generality, we consider  Poiseuille-type flows, $\boldsymbol{P}^R_{\Phi}=:{P}^R_{\Phi} (x_2)\bl{e}_1$ and $\boldsymbol{P}^L_{\Phi}=:{P}^L_{\Phi} (y_2)\bl{e}_1'$, of \eqref{NS} with the boundary condition \eqref{NBC} in $\mS_i$ ($i$ denotes $R$ or $L$), then we will find that
\be\label{POSS2}
\left\{
\begin{aligned}
&-\left(P^R_\Phi(x_2)\right)''=C_R,\q\q\q\text{in }\q (0,1),\\
&\left(P^R_\Phi(0)\right)'=\al P^R_\Phi(0),\q \left(P^R_\Phi(1)\right)'=-\al P^R_\Phi(1),\\
&\int_{0}^{1}P^R_\Phi(x_2)dx_2=\Phi, \\
\end{aligned}
\right.
\ee
and
\be\label{POSS1}
\left\{
\begin{aligned}
&-\left(P^L_\Phi(y_2)\right)''=C_L,\q\q\q\text{in }\q (0,c_0),\\
&\left(P^L_\Phi(0)\right)'=\al P^L_\Phi(0),\q \left(P^L_\Phi(c_0)\right)'=-\al P^L_\Phi(c_0),\\
&\int_{0}^{c_0}P^L_\Phi(y_2)dy_2=\Phi, \\
\end{aligned}
\right.
\ee
where the constants $C_i$ are uniquely related to $\Phi$.  Direct computation shows that
\be\label{PF}
\left\{
\begin{split}
P_{\Phi}^R(x_2)&=\f{6\Phi}{6+\al}\lt[\al(-x^2_2+x_2)+1\rt];\\
P_{\Phi}^L(y_2)&=\f{6\Phi}{c_0^2\lt(6+c_0\al\rt)}\lt[\al(-y^2_2+c_0y_2)+c_0\rt].\\
\end{split}
\right.
\ee
And
\be\label{EHP}
\bali
|(P^i_\Phi)'|\leq C\f{\al\Phi}{1+\al},\q\forall i\in\{L,R\},
\eali
\ee
where $C>0$ is a constant independent of both $\al$ and $\Phi$.

The main objective of this paper is to study the solvability of the following generalized Leray's problem: For a given flux $\Phi$, to find a pair $(\bl{u},p)$ such that
\be\label{GL1}
\lt\{
\bali
&\bl{u}\cdot\na \bl{u}+\na p-\Dl \bl{u}=0,\q \na\cdot \bl{u}=0,\q \hskip .1cm\text{in}\q \mathcal{S}, \\
&2(\mathbb{S}\bl{u}\cdot\boldsymbol{n})_{\mathrm{tan}}+\alpha \bl{u}_{\mathrm{tan}}=0,\q \bl{u}\cdot\boldsymbol{n}=0, \q \text{on}\q \p\mathcal{S},\\
\eali
\rt.
\ee
with
\be\label{GL2}
\int^{1}_{0}u_1(x_1, x_2) d x_2=\Phi,\q\text{for any}\q x\in\mathcal{S}_R
\ee
and
\be\label{GL3}
\bl{u}\rightarrow \boldsymbol{P}^i_{\Phi},\q \text{as}\q |x|\rightarrow \i \ \text{ in }\ \mS_i.
\ee

To prove the existence of the above generalized Leray's problem, we first introduce a weak formulation. Multiplying \eqref{GL1}$_1$ with  $\bl{\varphi}\in\bl{X}$ and integration by parts, by using the boundary condition \eqref{GL1}$_2$, we can obtain
\be\label{weaksd}
\begin{split}
&2\int_{\mS}\mathbb{S}\bl{u}:\mathbb{S}\bl{\varphi} dx+\al\int_{\p\mS}\bl{u}_{\mathrm{tan}}\cdot\bl{\varphi}_{\mathrm{tan}}dS+\int_{\mS}\bl{u}\cdot\nabla \bl{\varphi} \cdot \bl{u} dx=0,\q \text{for all }\ \bl{\varphi}\in \bl{X}.
\end{split}
\ee
Now we define the weak solution of the generalized Leary's problem:

\begin{definition}\label{weaksd1}
A vector $\bl{u}:\mS\to \bR^2$ is called a {\it weak} solution of the generalized Leray problem \eqref{GL1} to \eqref{GL3} if and only if
\begin{itemize}
\setlength{\itemsep}{-2 pt}
\item[(i).] $\bl{u}\in \mathcal{H}_{\sigma,{\mathrm{loc}}}(\overline{\mS})$;

\item[(ii).] $\bl{u}$ satisfies \eqref{weaksd};

\item[(iii).] $\bl{u}$ satisfies \eqref{GL2} in the trace sense;

\item[(iv).] $\bl{u}-\bl{P}^i_{\Phi}\in {H}^1(\mS_i)$, for $i=L,R$.
\end{itemize}
\end{definition}

\qed

\begin{remark}
The weak solution also satisfies a generalized version of \eqref{GL3}. Actually, it follows from the trace inequality (\cite[Theorem II.4.1]{Galdi2011}) that for any $x\in\mS_R$:
\be\label{WK111}
\bali
&\int^1_0|\bl{u}(x_1,x_2)-\bl{P}^R_{\Phi}(x_2)|^2dx_2\leq C\|\bl{u}-\bl{P}^R_{\Phi}\|^2_{H^1([x_1,+\i)\times[0,1])},\q\forall x_1>0,
\eali
\ee
where the constant $C$ is independent of $x_1$. Using $(iv)$ in Definition \ref{weaksd1}, the estimate \eqref{WK111} implies that
\[
\int^1_0|\bl{u}(x_1,x_2)-\bl{P}^R_{\Phi}(x_2)|^2dx_2 \to0,\quad\text{as $x_1\to\infty$}.
\]
The case in $\mS_L$ is similar.
\end{remark}

\qed

The following result shows that for each weak solution we can associate a corresponding pressure field. See the proof in Section \ref{SEC3627} below.

\begin{lemma}\label{pressure}
Let $\bl{u}$ be a weak solution to the generalized Leray's problem defined above. Then there exists a scalar function $p\in L^2_{{\mathrm{loc}}}(\overline{\mS})$ such that
\[
\int_{\mS}\na\bl{u}:\na\bl{\psi} dx+\int_{\mS}\bl{u}\cdot\nabla \bl{u}\cdot \bl{\psi}dx=\int_{\mS} p\na\cdot\bl{\psi} dx
\]
holds for any $\bl{\psi}\in C^\i_c(\mS; \bR^2)$.
\end{lemma}

\qed

\subsection{Main results}

\q\ Now we are ready to state the main theorems of this paper. The first one is existence and uniqueness of the weak solution, the second one addresses the regularity and decay estimates of the weak solution.

\begin{theorem}\label{PRO1.2}
Let $0\leq \al\leq +\i$ be the friction coefficient given in \eqref{NBC}. Assume that $\mS$ is the aforementioned smooth strip. Then there exists $C_0>0$, independent of $\al$, such that

\begin{itemize}
\item[(i)] if
\be\label{COND1}
\f{\al\Phi}{1+\al}\leq C_0,
\ee
then the 2D generalized Leray's problem \eqref{GL1}-\eqref{GL2}-\eqref{GL3} has a weak solution $(\bl{u},p)\in \mathcal{H}^1_{\sigma,{\mathrm{loc}}}(\overline{\mS})\times L^2_{{\mathrm{loc}}}(\overline{\mS})$ satisfying
\be\label{sesti}
\sum_{i=L,R}\|\bl{u}-\boldsymbol{P}^i_\Phi\|_{H^1(\mS_i)}\leq \Phi e^{C \Phi},
\ee
for some constant $C$ independent of $\al$.

\item[(ii)] Moreover, if $\t{\bl{u}}$ is another weak solution of \eqref{GL1}-\eqref{GL2}-\eqref{GL3} with the flux $\Phi\leq C_0$, and satisfies that for $\zeta\to\i$,
\be\label{GROWC}
\|\na \t{\bl{u}}\|_{L^2(\mS_\zeta)}=o\lt(\zeta^{3/2}\rt).
\ee
then $\t{\bl{u}}\equiv \bl{u}$.
\end{itemize}

\end{theorem}

\qed

\begin{remark} Here we give several remarks.

\begin{itemize}

\item In the existence result (i), noticing that the flux at the cross section $\Phi$ can be relatively large when $\al<1$ is small, since one only needs $\Phi\leq 2C_0\al^{-1}$. Here $\al=0$ means the flux $\Phi$ can be arbitrarily large.

\item The limiting case $\al=0$ (i.e., the total slip situation) has already been considered in \cite{Mucha2003}, where an extra geometry restriction on the shape of the strip was imposed and the uniqueness was not considered there.

    \item The limiting case $\al=+\i$ corresponds to the famous Leray's problem with the no-slip boundary condition which has been investigated for a long period of time. See a systematic review and study in \cite[Chapter XIII]{Galdi2011}.
\item From the uniqueness result in Theorem \ref{PRO1.2}, we see that uniqueness can be only guaranteed by assuming that $\Phi$ is small enough, independent of the scale of $\al$. Actually uniqueness of the weak solution is a more complicated problem than existence. See some discussion and non-uniqueness results in \cite[Chapter XII.2]{Galdi2011} for the stationary 2D exterior problem.

\end{itemize}

\end{remark}

\qed

The following Theorem gives the smoothness and the asymptotic behavior of a weak solution, which decays exponentially to the Poiseuille flow $\boldsymbol{P}^i_{\Phi}$ at each $\mS_i$ as $|x|$ tends  to infinity. Only the partial smallness condition \eqref{COND1} is imposed.

%

\begin{theorem}\label{THMSA}
Let $\bl{u}$ be a weak solution stated in the item $(i)$ of Theorem \ref{PRO1.2}. Then
\[
\bl{u}\in C^\i(\overline{\mS})
\]
such that: For any  $m\geq 0$,
\be\label{HODEST}
\sum_{i=L,R}\|\na^m(\bl{u}-\boldsymbol{P}^i_\Phi)\|_{L^2(\mS_i)}+\|\na^m \bl{u}\|_{L^2(\mS_0)}\leq C_{\Phi,m}.
\ee
Meanwhile, the following pointwise decay estimates hold for sufficiently large $|x|$:
\be\label{pointdecay}
\begin{split}
|\na^m(\bl{u}-\boldsymbol{P}^L_{\Phi})(y)|&\leq C_{\Phi,m}  e^{\sigma_{\Phi,m}y_1},\q\text{for all}\q y_1<<-1;\\
|\na^m(\bl{u}-\boldsymbol{P}^R_{\Phi})(x)|&\leq C_{\Phi,m}  e^{-\sigma_{\Phi,m}x_1},\q\text{for all}\q x_1>>1.
\end{split}
\ee
Here $C_{\Phi,m}$ and $\sigma_{\Phi,m}$ are positive constants depending only on $\Phi$ and $m$.
\end{theorem}

\qed

Also the corresponding pressure $p$ enjoys estimates akin to \eqref{HODEST} and \eqref{pointdecay}. See Remark \ref{RMMKK43}.
\begin{remark}
After our paper was posted on arXiv, we were informed by Professor Chunjing Xie that their group are also considering 2D Leray's problem with Navier-slip boundary and two manuscripts on this topic are finished. We are grateful for their kindness  of sending us their manuscripts. After checking their manuscripts, though there are partial overlaps of results, the proof  between theirs and ours differs in many aspects. Since our two groups' manuscripts are nearly posted at the same time, we believe that we independently solve this 2D Leray's problem at almost the same time. Reader can refer their works in \cite{SWX:2022ARXIV1, SWX:2022ARXIV2} for more details.
\end{remark}
\subsection{Influence of the friction coefficient for the well-posedness}

\q\ Unlike the 3D generalized Leray's problem with the Navier-slip boundary condition in our recent work \cite{LPY:2022ARXIV}, the friction coefficient $\al$ plays an important role for the well-posedness in the 2D problem. Some interesting results different from the 3D problem are presented as follows:
\begin{itemize}
\setlength{\itemsep}{-3 pt}
\item[(i).]Largeness of the flux $\Phi$ when we show the existence, regularity and asymptotic behavior of the constructed $H^1$ weak solution.

\item[(ii).] The $\al-$independence of all the estimates in Theorem \ref{PRO1.2} and Theorem \ref{THMSA} indicates that our results are uniform with the friction coefficient $\al$ and can be applied to the limiting cases $\al=0$ (total slip case) and $\al=\i$ (classical Leray's problem).

\end{itemize}

The main reason behind the above improvements is the validity of the Korn-type inequality ($L^2$ norm equivalence between $\na\bl{v}$ and $\mathbb{S}\bl{v}$) in the 2D strip domain $\mS$ as displayed in Lemma \ref{lem-korn} and Corollary \ref{cor-korn}, which fails in the 3D pipe as Remark \ref{3korn}.

\subsection{Difficulties, outline of the proof and related works}

\subsubsection*{Difficulties and corresponding strategies}

\q\ In two dimensional case, compared with the no-slip boundary condition, the main difficulties of the problem with Navier-slip boundary condition lie in the following:
\begin{itemize}
\setlength{\itemsep}{-3 pt}
\item[(i).] For a given flux, construction of a smooth  solenoidal flux carrier, satisfying the Navier-slip boundary condition, and equalling to the Poiseuille flow at a large distance;

\item[(ii).] Achieving Poincar\'e-type inequalities and Korn-type inequalities in the distorted strip $\mS$.

\end{itemize}
In order to overcome difficulties listed above, our main strategies are as follows:
\begin{itemize}

\setlength{\itemsep}{-3 pt}

\item[(i).] In order to construct the flux carrier, we first introduce a uniform curvilinear coordinates transform near the boundary: $T:\, x\rightarrow (s,t)$ to smoothly connect the two semi-infinite long straight strips and the middle distorted part. Under this  curvilinear coordinates, a thin strip near the boundary of the middle distorted part is straightened. Under the curvilinear coordinates $(s,t)$, the flux carrier is constructed to smoothly connect the Poiseuille flows $\bl{P}_\Phi^L$ and $\bl{P}_\Phi^R$ at far field with a compact supported divergence-free vector $\sigma'_\ve(t)\bl{e}_s$ in $\mS_0$. In the intermediate parts, they can be glued smoothly, and the divergence-free property together with the Navier-slip boundary condition keep valid.

\item[(ii).] The Poincar\'e inequality and Korn's inequality play important roles during the proof. For the no-slip boundary condition, Poincar\'e inequality can be applied directly by using zero boundary condition. However, in the case of the Navier-slip boundary condition, the Poincar\'e inequality is not obvious in the middle distorted part $\mS_0$. First, in $\mS_L$ or $\mS_R$, after subtracting the constant flux, $\bl{v}\cdot\bl{e_1}$ (or $\bl{v}\cdot\bl{e_1}'$) has zero mean value in any cross line, and then combining the impermeable boundary condition, which indicates that $\bl{v}\cdot\bl{e_2}=0$ (or $\bl{v}\cdot\bl{e_2}'=0$) on the boundary, we can achieve Poincar\'e inequality in the straight strips. Based on the result in the straight strip, we derive the Poincar\'e inequality in $\mS_0$ by the trace theorem and a 2D Payne's identity \eqref{Payne3D}. See Lemma \ref{TORPOIN}. The $\al$-independence of constants during the proof of main theorems is creditable to Korn's inequalities in 2D strips. The Korn's inequality is proved via a contradiction argument, which is given in Section \ref{PRE}, and is highly dependent on the compactness of the curvature of the boundary. It is not valid in the 3D case. See a counterexample in Remark \ref{3korn}.

\end{itemize}

\subsubsection*{Outline of the proof}

\q\ The existence and uniqueness of the solution will be given in Section \ref{SEU}. Before proving the existence theorem, a smooth solenoidal flux carrier $\bl{a}$ will be carefully constructed under the help of the  uniformly curvilinear coordinates near the boundary $\p\mS_{-}$. By subtracting this smooth flux carrier, the existence problem of \eqref{GL1}-\eqref{GL2}-\eqref{GL3} is reduced to a related one that the solution approaches zero at spacial infinity, which can be handled by the standard Galerkin method.

The main idea of proving the uniqueness is applying Lemma \ref{LEM2.3}, which was originally announced in reference \cite{Lady-Sol1980} and used to prove the uniqueness of the Leray's problem with no-slip boundary. Although our idea originates from \cite{Lady-Sol1980}, there are many differences from the previous literature. Some estimates of the present manuscript is much more complicated, involving the Poincar\'e inequality in a distorted strip as shown in Lemma \ref{TORPOIN} and Korn's inequality in Lemma \ref{lem-korn}.

Proofs of the asymptotic behavior and smoothness of weak solutions are given in Section \ref{SECH}. The main idea in deriving the exponential decay of $H^1$ weak solutions is to derive a first order ordinary differential inequality for the $L^2$ gradient integration in domain $\mS\backslash\mS_\zeta$. For the global estimates of higher-order norms, by using a ``decomposing-summarizing" technique, $H^1$ estimate of the vorticity in $\mS$ will be obtained, and then Biot-Savart law indicates $H^2$ global estimate of the solution. Using the bootstrapping argument, higher-order global estimates then follow. This also leads to the higher-order exponential decay estimates, by utilizing the $H^1$-decay estimate and interpolation inequalities.

\subsubsection*{Some related works}

\q\  The well-posedness study of the stationary Navier-Stokes equations in an infinite long pipe (or an infinite strip in the 2D case) with no-slip boundary condition and toward the Poiseuille flow laid down by Ladyzhenskaya in 1950s \cite{Ladyzhenskaya:1959UMN,Ladyzhenskaya:1959SPD}, in which the problem was called \emph{Leray's problem}. Later by reducing the problem to the resolution of a variational problem, Amick \cite{Amick:1977ASN, Amick:1978NATMA} obtained the existence result of the Leray's problem with small flux, but the uniqueness was left open. For the planar flow, Amick--Fraenkel \cite{Amick-Fraenkel1980} studied the Leray's problem in various types of stripes distinguished by their properties at infinity. An approach to solving the uniqueness of small-flux solution via energy estimate was built by Ladyzhenskaya-Solonnikov \cite{Lady-Sol1980}, in which authors also addressed the existence and asymptotic behavior results. See \cite{AP:1989SIAM, HW:1978SIAM, Pileckas:2002MB} for more related conclusion, also \cite[Chapter XIII]{Galdi2011} for a systematic review to the Leray's problem with the no-slip boundary condition. Recently Yang-Yin \cite{YY:2018SIAM} studied the well-posedness of weak solutions to the steady non-Newtonian fluids in pipe-like domains. Wang-Xie in \cite{WX:2019ARXIV,Wang-Xie2022ARMA} studied the existence, uniqueness and uniform structural stability of Poiseuille flows for the 3D axially symmetric inhomogeneous Navier-Stokes equations in the 3D regular cylinder, with a force term appearing on the right hand of the equations.

The Navier-slip boundary condition was initialed by Navier \cite{Navier}. It allows fluid slip on the boundary with a scale proportional to its stress tensor. Different from the no-slip boundary, the Leray's problem with the Navier-slip boundary condition requires much more complicated mathematical strategies. \cite{Mucha2003, Mucha:2003STUDMATH, Konie:2006COLLMATH} studied the solvability of the steady Navier-Stokes equations with the perfect Navier-slip condition ($\al=0$). In this case, the solution approaches to a constant vector at the spatial infinity. Authors in \cite{Amrouche2014,Ghosh:2018PHD,AACG:2021JDE} studied the properties of solutions to the steady Navier-Stokes equations with the Navier-slip boundary in bounded domains. Wang and Xie \cite{WX:2021ARXIV2} showed the uniqueness and uniform structural stability of Poiseuille flows in an infinite straight long pipe with the Navier-slip boundary condition. Authors of the present paper studied the related 3D Leray's problem with the Naiver-slip boundary condition \cite{LPY:2022ARXIV} under more strict smallness of the flux than the recent paper on 2D case. They also proved the characterization of bounded smooth solutions for the 3D axially symmetric Navier-Stokes equations with the perfect Navier-slip boundary condition in the infinitely long cylinder \cite{LP:2021ARXIV}.

This paper is arranged as follows: In Section \ref{PRE}, some preliminary work are contained, in which a uniform curvilinear coordinate near the boundary will be introduced and the Navier-slip boundary condition will be written under this curvilinear coordinate frame, and some useful lemmas will be presented. We will concern the existence and uniqueness results in Section \ref{SEU}. Finally, we focus on the higher-order regularity and exponential decay properties of the solution in Section \ref{SECH}.

\section{Preliminary}\label{PRE}

First, we introduce a uniformly curvilinear coordinate near the boundary, which will help to construct the flux carrier. This curvilinear coordinates can be viewed as the straightening of the boundary in the distorted part $\mS_0$.  Under this curvilinear coordinates, the Navier-slip boundary condition in \eqref{NBC} on the boundary of $\mS_0$ will share almost the same form as that in the semi-infinite straight part of $\mS_L$ and $\mS_R$. See \eqref{NBCM} below.

\subsection{On the uniformly curvilinear coordinates near the boundary}
To investigate the delicate feature of the Navier-slip boundary condition, also to construct the flux carrier in the distorted part of the strip, one needs to parameterize the boundary of $\mS$. Recalling that
\[
\p\mS=\p\mS_+\cup\p\mS_-\cup\p\mS_{Ob},
\]
where $\p\mS_{\pm}$ are upper and lower boundary portions of $\mS$, while $\p\mS_{Ob}$ denotes the union of boundaries of obstacles in the middle of the strip. For convenience, we only parameterize $\p\mS_-$ since the others are similar. Besides, under this parameterised curvilinear coordinates, we will construct the divergence-free flux carrier, which is supported in
\[
\mS_-(\dl):=\{x\in\overline{\mS}:\,\mathrm{dist }(x,\p\mS_-)<\dl\},
\]
for some suitably small $\dl$ in the next section.
\begin{figure}[H]\label{FIG1}
\centering
\includegraphics[scale=0.3]{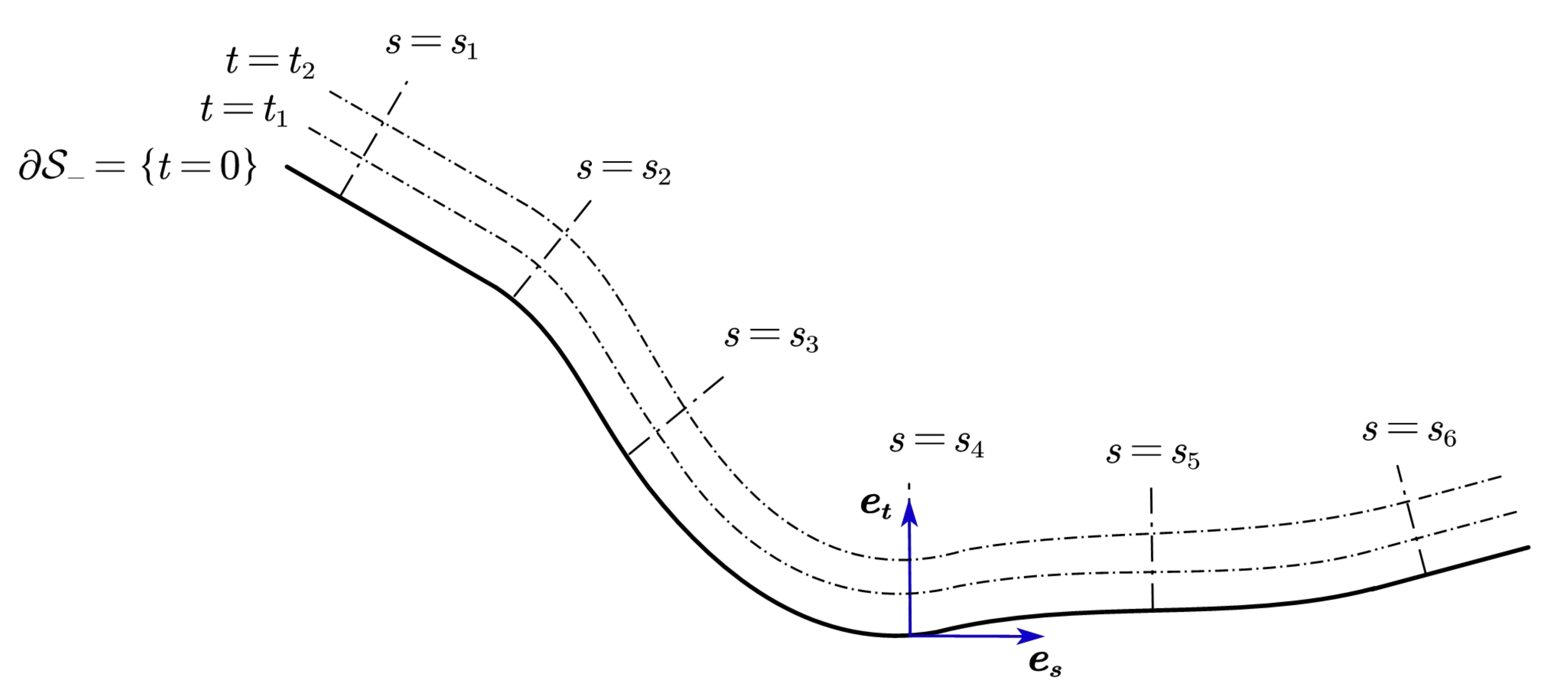}\\
\caption{A curvilinear coordinate system $(s,t)$ near the boundary portion $\p\mS_-$.}\label{FIG1}
\end{figure}

Denoting \be\label{PS-}
\p\mS_-=\{\bl{\mathfrak{b}}(s)=\left(\mathfrak{b}_1(s),\mathfrak{b}_2(s)\right)\in\mathbb{R}^2:\,s\in\mathbb{R}\},
\ee
where $\mathfrak{b}_1(s),\mathfrak{b}_2(s)$ are smooth functions of $s$. Without loss of generality, we suppose the parameter $s\in\mathbb{R}$ being the \emph{arc length parameter} of $\p\mS_-$, so that
\[
|\bl{\mathfrak{b}}(s)|=\sqrt{(\mathfrak{b}_1'(s))^2+(\mathfrak{b}_2'(s))^2}\equiv 1,\q\forall s\in\mathbb{R}.
\]
By the definition of $\mS$ given in Section \ref{SEC1}, $\p\mS_-$ lies on part of straight lines $\{y_2=0\}$ or $\{x_2=0\}$ except a compact distorted part in the middle, and there exists $s_0>0$ such that
\be\label{para}
\p\mS_-\cap\overline{\mS_0}=\{\bl{\mathfrak{b}}(s)=\left(\mathfrak{b}_1(s),\mathfrak{b}_2(s)\right)\in\mathbb{R}^2:\,s\in[-s_0,s_0]\}.
\ee
This indicates $\bl{\mathfrak{b}}(s_0)=O$ and $\bl{\mathfrak{b}}(-s_0)=O'$. Meanwhile, all ``obstacles" inside $\mS$ are away from it.

Because of the compact distortion, $\p\mS_-$ must satisfy the following condition:
\begin{condition}[Uniform interior sphere]\label{COND1.1}
	For any point $z\in\p\mS_-$, there exists a disk $K_{z}$, with its radius being $R_{z}$, such that
	\[
	\overline{K_{z}}\cap(\mathbb{R}^2-\mS)=\{z\}.
	\]
	Meanwhile, there exists $\dl>0$ such that
	\be\label{UISC}
	R_{z}\geq2\dl,\q\forall z\in\p\mS_-.
	\ee
\end{condition}

\qed

Due to the uniform interior sphere condition, for any $x\in\mS_-(\dl)$, there exists a unique point $z\in\p\mS_-$ such that $|x-z|=\mathrm{dist }(x,\p\mS_-)$. Recalling \eqref{PS-}, there exists a unique pair $(s,t)\in\mathbb{R}\times[0,\dl)$ such that $y=\bl{\mathfrak{b}}(s)$ and $t=\mathrm{dist }(x,z)$. In this way, the following mapping is one-to-one and well-defined:
\be\label{MAPP}
x\to(s,t),\q\forall x\in\mS_-(\dl).
\ee
Meanwhile, one has
\begin{lemma}\label{LEM1.2}
The mapping defined in \eqref{MAPP} is smooth.
\end{lemma}

\begin{proof}
By the construction of this mapping, one deduces
	\[
	x=y-\mathrm{dist }(x,z)\bl{n}_z,
	\]
	where $\bl{n}_y$ is the unit outer normal of $y\in\p\mS_-\,$. Since $y=\bl{\mathfrak{b}}(s)$, we define
	\[
	\bl{F}(s,t):=\bl{\mathfrak{b}}(s)-t\bl{n}_{\bl{\mathfrak{b}}(s)}\q\forall(s,t)\in\mathbb{R}\times[0,\dl)\,\,.
	\]
Clearly $\bl{F}$ is well-defined and smooth in $\mathbb{R}\times[0,\dl)$, and its Jacobian matrix writes
	\[
	J\bl{F}=\left(
	\begin{array}{cc}
		\mathfrak{b}_1'(s)-t\left(\f{d}{d s}\bl{n}_{\bl{\mathfrak{b}}(s)}\right)_1&\q-(\bl{n}_{\bl{\mathfrak{b}}(s)})_1\\[2mm]
		\mathfrak{b}_2'(s)-t\left(\f{d}{d s}\bl{n}_{\bl{\mathfrak{b}}(s)}\right)_2&\q-(\bl{n}_{\bl{\mathfrak{b}}(s)})_2\\
	\end{array}
	\right).
	\]
Here and below, $(\bl{X})_i$ with $i=1,2$ means the $i$-th component of the vector $\bl{X}$. Direct calculation shows
	\[
	\mathrm{det }(J\bl{F})=\left(\bl{\mathfrak{b}}'(s)-t\left(\f{d}{d s}\bl{n}_{\bl{\mathfrak{b}}(s)}\right)\right)\cdot\bl{n}_{\bl{\mathfrak{b}}(s)}^{\perp}=1-t\left(\f{d}{d s}\bl{n}_{\bl{\mathfrak{b}}(s)}\right)\cdot\bl{\mathfrak{b}}'(s),
	\]
	where
	\[
	\bl{n}_{\bl{\mathfrak{b}}(s)}^{\perp}=(-(\bl{n}_{\bl{\mathfrak{b}}(s)})_2,(\bl{n}_{\bl{\mathfrak{b}}(s)})_1)=\bl{\mathfrak{b}}'(s).
	\]
	This indicates that
	\[
	\mathrm{det }(J\bl{F})\geq1-t\f{1}{2\dl}>\f{1}{2},\q\forall(s,t)\in\mathbb{R}\times[0,\dl)
	\]
due to \eqref{UISC} so that $\f{1}{2\dl}$ can bound the curvature of $\p\mS_-\,$. Recalling the compactness of the distorted part, the lemma is claimed by the inverse mapping theorem.
\end{proof}

For any $x=(x_1,x_2)\in\mS_-(\dl)$, Condition \ref{COND1.1} and Lemma \ref{LEM1.2} above guarantee the following well-defined curvilinear coordinate system
\[
(s,t)=(s(x),t(x))\in\mathbb{R}\times[0,\dl).
\]
Geometrically, $t(x)$ is the distance of the given point $x\in\mS_-(\dl)$ to the boundary $\p\mS_-$, while $s(x)$ denotes the parameter coordinate of the unique point $y\in\p\mS_-$ such that $|x-y|=\mathrm{dist }(x,\p\mS_-)$. As it is shown in Figure \ref{FIG1}, we denote
\be\label{ESET}
\bl{e_s}=(e_{s1}\,,\,e_{s2})\q\text{and}\q\bl{e_t}=(e_{t1}\,,\,e_{t2})
\ee
are unit tangent vector of $s-$curves and $t-$curves, respectively. Meanwhile, they are all independent with variable $t\in[0,\dl)$. Clearly
\be\label{ESET}
\left\{
\begin{aligned}
\bl{e_s}&\equiv\bl{e_1};\\
\bl{e_t}&\equiv\bl{e_2},\\
\end{aligned}
\right.\q\q\forall x\in\overline{\mS_R};\q\q\left\{
\begin{aligned}
\bl{e_s}&\equiv\bl{e_1}';\\
\bl{e_t}&\equiv\bl{e_2}',\\
\end{aligned}
\right.\q\q\forall x\in\overline{\mS_L}.
\ee
Moreover,
\be\label{ENTS}
\nabla_x t=\left(\f{\p t}{\p x_1}\,,\,\f{\p t}{\p x_2}\right)\equiv\bl{e_t},
\ee
and there exists a smooth function $\gamma(s,t)>\gamma_0>0$ that
\be\label{ENTS1}
\nabla_x s=\left(\f{\p s}{\p x_1}\,,\,\f{\p s}{\p x_2}\right)=\gamma\bl{e_s}.
\ee
Thus by denoting
\[
\bl{D}:=\left(
\begin{array}{cc}
	\f{\p s}{\p x_1}&\f{\p s}{\p x_2}\\[2mm]
	\f{\p t}{\p x_1}&\f{\p t}{\p x_2}\\
\end{array}
\right),
\]
one derives
\[
\mathrm{det }\,\bl{D}=-\gamma\bl{e_s}\cdot\bl{e_t}^{\perp}=\gamma.
\]
by \eqref{ENTS} and \eqref{ENTS1}. Moreover, by calculating the inverse matrix of $\bl{D}$, one deduces
\[
\left(
\begin{array}{cc}
	\f{\p x_1}{\p s}&\f{\p x_1}{\p t}\\[2mm]
	\f{\p x_2}{\p s}&\f{\p x_2}{\p t}\\
\end{array}
\right)=\left(\begin{array}{cc}
	\f{1}{\gamma}e_{t2}  & -e_{s2}\\[2mm]
	-\f{1}{\gamma}e_{t1} & e_{s1}\\
\end{array}
\right),
\]
which indicates
\[
\left\{
\begin{aligned}
&\f{\p x}{\p s}=\f{\bl{e_s}}{\gamma};\\
&\f{\p x}{\p t}=\bl{e_t}.\\
\end{aligned}
\right.
\]

Since $|\bl{e_s}|=|\bl{e_t}|\equiv 1$ and $\bl{e_s}\cdot\bl{e_t}\equiv0$, there exists a bounded smooth function $\kappa=\kappa(s,t)\in\mathbb{R}$, which denotes the curvature of the boundary, such that
\[
\left\{
\begin{split}
\f{d\bl{e_t}}{ds}&=-\f{\kappa\bl{e_s}}{\gamma};\\[2mm]
\f{d\bl{e_s}}{ds}&=\f{\kappa\bl{e_t}}{\gamma},\\
\end{split}
\right.\q\q\forall(s,t)\in\mathbb{R}\times[0,\dl).
\]
By direct calculation, the divergence and curl of a vector field $\bl{w}=w_s(s,t)\bl{e_s}+w_t(s,t)\bl{e_t}$ writes
\be\label{DIV-CURL}
\begin{split}
\mathrm{div }\,\bl{w}&=\gamma\p_sw_s+\p_tw_t-\kappa w_t;\\
\mathrm{curl }\,\bl{w}&=\p_{x_2}w_1-\p_{x_1}w_2=\p_tw_s-\gamma\p_sw_t-\kappa w_s,
\end{split}
\ee
under this curvilinear coordinates.

To finish this subsection, let us focus on the Navier-slip boundary condition under the curvilinear coordinates. Writing
\[
\bl{u}=u_{s}\bl{e_s}+u_t\bl{e_t}.
\]
Then \eqref{NBC} enjoys the following simplified expression:
\be\label{NBCM}
\left\{
\begin{aligned}
&\f{\p u_s}{\p\bl{n}}=-\p_{t}u_{s}=\left(\kappa-\al\right)u_{s},\\
&u_t=0,\\
\end{aligned}
\right.\q\text{on}\q\p\mS_-.
\ee
See \cite[Proposition 2.1 and Corollary 2.2]{Watanabe2003} for a detailed calculation. Moreover, denoting $\mathfrak{w}=\p_{x_2}u_1-\p_{x_1}u_2$ and applying \eqref{DIV-CURL}$_2$, one has \eqref{NBCM}$_1$ is equivalent to
\[
\mathfrak{w}=\left(-2\kappa+\al\right)u_{s},\q\text{on}\q\p\mS_-.
\]

\subsection{The Poincar\'e inequality and the Korn's inequality}

The following Poincar\'e inequalities and Korn's inequality will play crucial role in the existence and uniqueness results when the no-slip boundary is replaced by the Navier-Slip boundary.

\begin{lemma}[Poincar\'e inequality in a straight strip]\label{POIN}
Let $\bl{g}=g_1\boldsymbol{e_1}+g_2\boldsymbol{e_2}$ be a $H^1$ vector field in the box domain $S:=[a,b]\times[c,d]$, and satisfies that
\bes
\int^d_c g_1(x_1,x_2)dx_2=g_2(x_1,x_2)\big|_{x_2=c,d}=0, \q  \forall\ x_1\in [a,b],
\ees
%
then we have the following

\be\label{poinfull}
\|\bl{g}\|_{L^2(S)}\leq C\|\p_{x_2} \bl{g}\|_{L^2(S)}.
\ee
where $C\ls |c-d|$ is a constant depending on the width of the strip.
\end{lemma}

\pf Since $g_1$ has zero mean and $g_2$ has zero boundary in the $x_2$ direction. The classical one dimensional Poincar\'e inequality leads to
\bes
\|\bl{g}(x_1)\|^2_{L^2_{x_2}([c,d])}\ls_{|d-c|} \|\p_{x_2} \bl{g}(x_1)\|^2_{L^2_{x_2}([c,d])}, \q \forall\ x_1\in [a,b]
\ees
Integration on $[a,b]$ with respect to $x_1$ variable indicates \eqref{poinfull}.

\qed

\begin{lemma}[Poincar\'e inequality in the torsion part]\label{TORPOIN}
Let $\zeta>1$ and $\bl{h}=h_1\boldsymbol{e_1}+h_2\boldsymbol{e_2}=\tilde{h}_1\boldsymbol{e_1}'+\tilde{h}_2\boldsymbol{e_2}'\in H^1(\mS_\zeta)$ be a divergence free vector with zero flux, that is
\[
\int_{0}^1{h}_1(x_1,x_2)dx_2=0\q \text{for any}\ x\in\mS_\zeta\cap\mS_R.
\]
If we suppose $\bl{h}\cdot\bl{n}\equiv 0$ on $\p\mS\cap\p\mS_\zeta$, where $\bl{n}$ is the unit outer normal vector on $\p\mS$, then the following Poincar\'e inequality holds:
\be\label{TORPIPEPOIN}
\|\bl{h}\|_{L^2(\mS_\zeta)}\leq C\|\na\bl{h}\|_{L^2(\mS_\zeta)}.
\ee
Here $C>0$ is a uniform constant, independent of $\zeta$.
\end{lemma}
\pf Integrating the following identity on $\mS_{1}$,
\be\label{Payne3D}
\sum^2_{i,j=1}\lt[\p_{x_i}(h_i x_j h_j)-\p_{x_i} h_ix_jh_j-|\bl{h}|^2-h_ix_j\p_{x_i} h_j\rt]=0,
\ee
one deduces
\be\label{ZJ1}
\bali
\int_{\mS_{1}} |\bl{h}|^2dx=&\underbrace{\sum^2_{i,j=1}\int_{\mS_{1}}\p_i(h_ix_j h_j)dx}_{J_1}-\sum^2_{i,j=1}\int_{\mS_{1}}\p_i h_ix_jh_jdx-\sum^2_{i,j=1}\int_{\mS_{1}} h_ix_j\p_i h_jdx.\\
\eali
\ee
Using the divergence theorem and the boundary condition $\bl{h}\cdot{\boldsymbol{n}}=0$ on $\p\mS_{1}\cap\p\mS$, we can obtain
\[
J_1=\sum^2_{j=1} \int_{\p\mS_1}({\boldsymbol{n}}\cdot \bl{h})x_jh_jdS=\int_{0}^1\left((x\cdot\bl{h})h_1\right)(1,x_2)dx_2-\int_{0}^{c_0}\left((x\cdot\bl{h})\tilde{h}_1\right)(-1,y_2)dy_2.
\]
Thus by \eqref{ZJ1} and the Cauchy-Schwarz inequality, we arrive at
\[
\begin{split}
\int_{\mS_{1}}|\bl{h}|^2dx\leq& \f{1}{2}\int_{\mS_{1}}|\bl{h}|^2dx+C\int_{\mS_{1}}|\na\bl{h}|^2dx\\
&+\left|\int_{0}^1\left((x\cdot\bl{h})h_1\right)(1,x_2)dx_2\right|+\left|\int_{0}^{c_0}\left((x\cdot\bl{h})\tilde{h}_1\right)(-1,y_2)dy_2\right|,
\end{split}
\]
which indicates
\be\label{PCC1}
\int_{\mS_{1}}|\bl{h}|^2dx\leq C\left(\int_{\mS_{1}}|\na\bl{h}|^2dx+\int_{0}^1\left|\bl{h}(1,x_2)\right|^2dx_2+\int_{0}^{c_0}\left|\bl{h}(-1,y_2)\right|^2dy_2\right).
\ee
Meanwhile, using the trace theorem in $\mS_{L,1}$, and Lemma \ref{POIN}, one derives
\be\label{PCC2}
\begin{split}
\int_{0}^{c_0}\left|\bl{h}(-1,y_2)\right|^2dy_2&\leq C\left(\int_{(-1,0)\times(0,c_0)}\left|\bl{h}\right|^2dy_1dy_2+\int_{(-1,0)\times(0,c_0)}\left|\na\bl{h}\right|^2dy_1dy_2\right)\\
&\leq C\int_{\mS_{L,1}}\left|\na\bl{h}\right|^2dx.
\end{split}
\ee
Similarly, one deduces that
\be\label{PCC3}
\int_{0}^{1}\left|\bl{h}(1,x_2)\right|^2dx_2\leq C\int_{\mS_{R,1}}\left|\na\bl{h}\right|^2dx.
\ee
Substituting \eqref{PCC2}--\eqref{PCC3} in the right hand side of \eqref{PCC1}, one deduces
\[
\int_{\mS_{1}}|\bl{h}|^2dx\leq C\int_{\mS_{1}}|\na\bl{h}|^2dx.
\]
Using Lemma \ref{POIN}, it is easy to see that
\[
\int_{\mS_\zeta\backslash\mS_{1}}|\bl{h}|^2dx\leq C\int_{\mS_\zeta\backslash\mS_{1}}|\na\bl{h}|^2dx.
\]
Combining the above two inequalities, we  finish the proof of \eqref{TORPIPEPOIN}.

\qed

After showing Lemma \ref{POIN} and Lemma \ref{TORPOIN}, one concludes the following Poincar\'e inequality in the whole infinite strip:

\begin{corollary}[Poincar\'e inequality in $\mS$]\label{CorPoin}
Let
\[
\bl{g}\in\mathcal{V}:=\left\{\bl{f}=(f_1,f_2)\in H^1(\mS):\,\left(\bl{f}\cdot\bl{n}\right)\big|_{\p\mS}=0,\,\mathrm{div }\,\bl{f}=0\right\},
\]
then the following Poincar\'e inequality holds:
\be\label{TORPIPEPOIN}
\|\bl{g}\|_{L^2(\mS)}\leq C\|\na\bl{g}\|_{L^2(\mS)}.
\ee
Here $C>0$ is a uniform constant.
\end{corollary}
\pf This is a direct conclusion by gluing results in Lemma \ref{POIN} and Lemma \ref{TORPOIN} together, after we have shown
\[
\int_{\mS\cap\{x_1=0\}}g_1dx_2=0
\]
holds unconditionally for $\bl{g}\in\mathcal{V}$. By the divergence theorem,
\[
\int_{\mS\cap\{x_1=\mathfrak{s}\}}g_1dx_2-\int_{\mS\cap\{x_1=0\}}g_1dx_2=\int_{\mS\cap\{0<x_1<\mathfrak{s}\}}\mathrm{div }\,\bl{g}dx-\int_{\p\mS\cap\{0<x_1<\mathfrak{s}\}}\left(\bl{g}\cdot\bl{n}\right)dS=0
\]
for any $\mathfrak{s}>0$. Thus if
\[
\left|\int_{\mS\cap\{x_1=0\}}g_1dx_2\right|=c_0>0,
\]
one deduces
\[
c_0\leq\int_{\mS\cap\{x_1=\mathfrak{s}\}}|g_1|dx_2\leq\left|\mS\cap\{x_1=\mathfrak{s}\}\right|^{1/2}\left(\int_{\mS\cap\{x_1=\mathfrak{s}\}}|g_1|^2dx_2\right)^{1/2}.
\]
This implies that, for any $\mathfrak{s}>Z_0$
\[
\int_{\mS\cap\{x_1=\mathfrak{s}\}}|g_1|^2dx_2\geq c_0^2,
\]
which results in a paradox with $\bl{g}\in H^1(\mS)$ .

\qed

Here goes the Korn's inequality in the truncated strip:

\begin{lemma}[Korn's inequality]\label{lem-korn}
Let $\mS_{\zeta}$ with $\zeta>0$ be the finite truncated strip given in \eqref{Szeta}. For any
\[
\bl{g}\in\mathcal{V}_\zeta:=\left\{\bl{f}=(f_1,f_2)\in H^1(\mS_{\zeta};\,\mathbb{R}^2):\,\left(\bl{f}\cdot\bl{n}\right)\big|_{\p\mS\cap \p\mS_{\zeta}}=0,\,\mathrm{div }\,\bl{f}=0,\,\int_{\mS\cap\{x_1=0\}}f_1dx_2=0\right\},
\]
there exists $C>0$, which is independent of $\bl{g}$ or $\zeta$, such that
\be\label{KN1.1}
\|\bl{g}\|^2_{H^{1}(\mS_{\zeta})}\leq C\|\mathbb{S}\bl{g}\|^2_{L^2(\mS_{\zeta})}+2\int_{\{x_1=\zeta\}}|\bl{g}||\na \bl{g}|dx_2+2\int_{\{y_1=-\zeta\}}|\bl{g}||\na \bl{g}|dy_2.
\ee
\end{lemma}
\pf Noting that
\be\label{KN655}
\begin{aligned}
\int_{\mS_{\zeta}}|\mathbb{S}\bl{g}|^{2} dx &=\frac{1}{2} \int_{\mS_{\zeta}} \sum_{i, j=1}^{2}\left(\p_{x_j}g_{ i}+\p_{x_i}g_{j}\right)^{2} dx \\
&=\int_{\mS_{\zeta}} \sum_{i, j=1}^{2}\left(\left(\p_{x_j}g_{i}\right)^{2}+\p_{x_j}g_{i}\p_{x_i}g_{j}\right) dx \\
&=\|\nabla \bl{g}\|_{L^{2}(\mS_{\zeta})}^{2}+\un{\int_{\mS_{\zeta}} \sum_{i, j=1}^{2}\p_{x_j}g_{i}\p_{x_i}g_{j}dx}_{K_1}.
\end{aligned}
\ee
For the last term of \eqref{KN655}, we can assume that $\bl{g}\in C^{2}(\mS_{\zeta})$  without loss of generality. By integration by parts, we get

\be\label{KN111}
\begin{aligned}
K_1=&-\int_{\mS_{\zeta}} g_{i}\p^2_{x_ix_j}g_jdx+\int_{\partial \mS_{\zeta}\cap\p\mS } \sum_{i, j=1}^{2} n_{j}g_{i} \p_{x_i}g_{j} dS+\int_{\partial \mS_{\zeta}\cap \{x_1=\pm \zeta\} } \sum_{i, j=1}^{2} n_{j}g_{i} \p_{x_i}g_{j} dx_2 \\
=& \int_{\mS} (\mathrm{div}\,\bl{g})^2dx-\int_{\partial \mS_{\zeta}\cap\p\mS} \sum_{i, j=1}^{2} n_{i}g_{i}\p_{x_j}g_{j} dS+\int_{\partial \mS_{\zeta}\cap\p\mS} \sum_{i, j=1}^{2} n_{j}g_{i} \p_{x_i}g_{j} dS\\
&-\int_{\{x_1=\zeta\}} \sum_{i, j=1}^{2} n_{i}g_{i}\p_{x_j}g_{j} dx_2-\int_{\{y_1=-\zeta\}} \sum_{i, j=1}^{2} n_{i}g_{i}\p_{x_j}g_{j} dy_2\\
&+\int_{\{x_1=\zeta\}} \sum_{i, j=1}^{2} n_{j}g_{i} \p_{x_i}g_{j} dx_2+\int_{\{y_1=-\zeta\}} \sum_{i, j=1}^{2} n_{j}g_{i} \p_{x_i}g_{j} dy_2.
\end{aligned}
\ee
The first, fourth and fifth terms on the far right of above equation vanish owing to the divergence-free property of $\bl{g}$, and the second one also vanishes because $\left.\bl{g} \cdot \bl{n}\right|_{\partial \mS}=0$. Meanwhile, the condition $\left.\bl{g} \cdot \bl{n}\right|_{\partial \mS}=0$ also implies that, at the boundary,

\be\label{KN222}
\sum_{i=1}^{2}g_i\p_{x_i}(\bl{g}\cdot \bl{n})=0 \quad \text { or } \quad \sum_{i,j=1}^{2} g_i\p_{x_i}g_{j} n_{j}=-\sum_{i,j=1}^{2} g_ig_{j} \p_{x_i}n_{j}.
\ee
Using \eqref{KN111} to \eqref{KN222}, we get

\[
\bali
K_1=&-\int_{\p\mS\cap\p\mS_\zeta }\kappa(x)|\bl{g}|^2dS+\int_{\{x_1=\zeta\}} \sum_{i, j=1}^{2} n_{j}g_{i} \p_{x_i}g_{j} dx_2+\int_{\{y_1=-\zeta\}} \sum_{i, j=1}^{2} n_{j}g_{i} \p_{x_i}g_{j} dy_2.
\eali
\]
where $\kappa(x)$ is the curvature of the boundary $\p\mS$. By definition of $\mS$, we have $\kappa(x) \equiv 0$ on $\p\mS\backslash\p\mS_0$, and $\bl{n}=\bl{e_1}$ on $\p\mS_{\zeta} \cap \{x_1=\zeta\}$, while $\bl{n}=-\bl{e_1}'$ on $\p\mS_{\zeta}\cap\{y_1=-\zeta\}$. This guarantees that
\be\label{KN699}
|K_1|\leq\int_{\p\mS\cap\p\mS_0}|\kappa(x)||\bl{g}|^2dS+\int_{\{x_1=\zeta\}} |\bl{g}| |\na \bl{g}| dx_2+\int_{\{y_1=-\zeta\}} |\bl{g}| |\na \bl{g}| dy_2.
\ee
Substituting \eqref{KN111}--\eqref{KN699} in \eqref{KN655}, one concludes
\[
\|\nabla \bl{g}\|_{L^{2}(\mS_\zeta)}^{2}\leq \int_{\mS_\zeta}|\mathbb{S}\bl{g}|^{2} dx+\int_{\p\mS\cap\p\mS_0}|\kappa(x)||\bl{g}|^2dS+\int_{\{x_1=\zeta\}} |\bl{g}| |\na \bl{g}| dx_2+\int_{\{y_1=-\zeta\}} |\bl{g}| |\na \bl{g}| dy_2.
\]
Noting that $\|\kappa(x)\|_{L^\i(\p\mS)}$ is uniformly bounded due to the smoothness of $\p\mS$ and combining with the Poincar\'e inequalities in Lemma \ref{POIN} and Lemma \ref{TORPOIN}, one deduces that there exists $C>0$ that
$$
\|\bl{g}\|_{H^1(\mS_\zeta)}^{2} \leq C\left(\int_{\mS_\zeta}|\mathbb{S}\bl{g}|^{2} dx+\|\bl{g}\|_{L^{2}(\p\mS\cap\p\mS_0)}^{2}\right)+\int_{\{x_1=\zeta\}} |\bl{g}| |\na \bl{g}| dx_2+\int_{\{y_1=-\zeta\}} |\bl{g}| |\na \bl{g}| dy_2.
$$

To finish the proof, one only needs to show that there exists $C>0$ such that:
\be\label{KN611}
\|\bl{g}\|_{L^{2}(\p\mS\cap\p\mS_0)}^{2} \leq\frac{1}{2 C}\|\bl{g}\|_{H^1(\mS_{\zeta})}^{2}+C\int_{\mS_{\zeta}}|\mathbb{S}\bl{g}|^{2} dx .
\ee
We prove this by the method of contradiction. If a number the above $C$ does not exist, then there exists a bounded sequence $\left\{\bl{g}_{m}\right\}_{m=0}^{\infty}\subset\mathcal{V}_{\zeta}$ such that
\[
\left\|\bl{g}_{m}\right\|_{L_{2}(\p\mS\cap\p\mS_0)}^{2} \geq \frac{1}{2 C_{1}}\left\|\bl{g}_{m}\right\|_{H^1(\mS_{\zeta})}^{2}+m \int_{\mS_{\zeta}}|\mathbb{S}\bl{g}_{m}|^{2} dx.
\]
Denoting $\bl{h}_{m}=\bl{g}_{m} /\left\|\bl{g}_{m}\right\|_{L^{2}(\p\mS\cap\p\mS_0)}$, one deduces that

\be\label{KN613}
\left\|\bl{h}_{m}\right\|_{L^{2}(\p\mS\cap\p\mS_0)}=1 \quad \text { and } \quad m \int_{\mS_{\zeta}}|\mathbb{S}\bl{h}_{m}|^{2} dx \leq 1.
\ee
Since the sequence $\left\{\bl{g}_{m}\right\}$ is bounded in $\mathcal{V}_{\zeta}$, we can choose a subsequence $\left\{\bl{h}_{m_{k}}\right\}_{k=0}^{\infty}$ which is weakly convergent in $H^1(\mS_{\zeta})$ and strongly in $L^{2}(\p\mS\cap\p\mS_0)$ to a vector $\bl{h}_{*} \in \mathcal{V}_{\zeta}$. Particularly,
\[
\mathbb{S}\bl{h}_{m_{k}}\to\mathbb{S}\bl{h}_*,\q\text{weakly in }L^2(\mS_{\zeta}).
\]
By \eqref{KN613}, one knows
$$
\int_{\mS_{\zeta}}|\mathbb{S}\bl{h}_{m_{k}}|^{2} dx \leq \frac{1}{m_{k}} \rightarrow 0,\q\text{as}\q k\to\i.
$$
Thus one deduces
\[
\int_{\mS_{\zeta}}|\mathbb{S}\bl{h}_{*}|^{2} \mathrm{~d} x\leq\liminf_{k\to\i}\int_{\mS_{\zeta}}|\mathbb{S}\bl{h}_{m_{k}}|^{2} dx=0,
\]
by the Fatou's lemma for weakly convergent sequences. This concludes $\mathbb{S}\bl{h}_{*}\equiv0$ in $\mS_{\zeta}$. It is well known that $\bl{h}_*$ has the form $\bl{h}_*=Ax+B$ (see \cite[\S 6]{KO:1988RMS}), where $A$ is a constant skew-symmetric matrix with constant entries and $B$ is a constant vector, that is,
\begin{equation*}
	\bl{h}_*=
	\begin{pmatrix}
		0 & -a\\
		a & 0
	\end{pmatrix}
	\begin{pmatrix}
		x_1 \\
		x_2
	\end{pmatrix}
	+\begin{pmatrix}
		b_1 \\
		b_2
	\end{pmatrix}=
	\begin{pmatrix}
		-a\,x_2+b_1	 \\
		a\,x_1+b_2
	\end{pmatrix}\,,
\end{equation*}
where $a\,,b_i$ ($i=1\,,2$) are some constants. However, by the boundary condition $\bl{h}_*\cdot\bl{n}=0$ holds everywhere on $\p\mS\cap\p\mS_{\zeta}$, one has
\[
(h_*)_2=ax_1+b_2\equiv0,\q\text{for all}\q 0<x_1<\zeta,
\]
which indicates $a=b_2\equiv 0$. This indicates $(h_*)_1=b_1$ and thus $b_1=0$ due to
\[
\int_0^1(h_*)_1(x_1,x_2)dx_2=0,\q\text{for all}\q 0<x_1<\zeta.
\]
Therefore one concludes $\bl{h}_*\equiv0$ in $\mS_{\zeta}$. However, this creates a paradox to the fact
\[
\|\bl{h}_*\|_{L^2(\p\mS\cap\p\mS_0)}=1
\]
coming from \eqref{KN613}. This indicates the validity of \eqref{KN611} and therefore one concludes \eqref{KN1.1}.

\qed

If we replace the truncated strip with the infinite strip $\mS$, the result in Lemma \ref{lem-korn} will be simpler with boundary term integrations on the segments $\{x_1=\zeta\}$ and $\{y_1=-\zeta\}$ disappearing. We have the following Corollary.

\begin{corollary}\label{cor-korn}
Let $\mS$ be the infinite strip given in the previous section. For any $\bl{g}\in \mathcal{V}$, there exists $C>0$, which is independent of $\bl{g}$, such that
\be\label{KN1.1plus}
\|\bl{g}\|_{H^{1}(\mS)}\leq C\|\mathbb{S}\bl{g}\|_{L^2(\mS)}.
\ee
\end{corollary}

\qed

\begin{remark}\label{3korn}
Here let us give a brief explanation why this Korn's inequality fails to be valid in a 3D infinite pipe. Consider the vector
\[
\bl{w}=(-\xi(x_3)x_2\,,\,\xi(x_3)x_1\,,\,0)
\]
given in the cylindrical pipe $\mathcal{D}=B\times\mathbb{R}$, where $B$ is the unit disk in $\mathbb{R}^2$, and $\xi$ is a smooth cut-off function that:
\[
\xi(x_3)=\left\{
\begin{aligned}
1\,,&\q\q x_3\in[-R,R];\\
0\,,&\q\q x_3\in\mathbb{R}\backslash(-R-1,R+1),
\end{aligned}
\right.
\]
with
\[
|\xi'(x_3)|\leq 2,\q\text{for any}\q x_3\in(-R-1,R)\cup(R,R+1).
\]
One notices that $\bl{w}$ is divergence-free and it satisfies $\bl{w}\cdot\bl{n}\equiv0$ on $\p B\times\mathbb{R}$, also its flux in the cross section $B\times\{x_3=0\}$ is zero.

For the convenience of calculation, we introduce the cylindrical coordinates:
\[
\boldsymbol{e_r}=(\frac{x_1}{r},\frac{x_2}{r},0),\quad \boldsymbol{e_\th}=(-\frac{x_2}{r},\frac{x_1}{r},0),\quad \boldsymbol{e_z}=(0,0,1),
\]
and we find
\[
\bl{w}=\xi(z)r\bl{e_\th}.
\]
Using equation (A.4) in \cite{LP:2021ARXIV}, one finds
\[
\mathbb{S}\bl{w}=\f{1}{2}\xi'(z)r\left(\bl{e_\th}\otimes\bl{e_z}+\bl{e_z}\otimes\bl{e_\th}\right).
\]
This indicates
\[
\int_{\mathcal{D}}|\mathbb{S}\bl{w}|^2dx=O(1)
\]
which is independent with $R$. On the other hand
\[
\int_{\mathcal{D}}|\nabla\bl{w}|^2dx\geq\int_{\mathcal{D}}|\p_2w_1|^2dx\geq 2\pi R.
\]
Noting that $R>0$ is arbitrary, one could not find a uniform constant $C>0$ such that a ``3D version" \eqref{KN1.1} or \eqref{KN1.1plus} holds.
\end{remark}

\qed

\begin{remark}
In the 3-dimensional case, the curvature of the domain boundary $\kappa(x)$ no longer has compact support. In this case one cannot find a subsequence $\left\{\bl{h}_{m_{k}}\right\}_{k=0}^{\infty}$ which is strongly convergent in $L^{2}(\p\mathcal{D})$ to a vector $\bl{h}_{*}$. That is why our method in the proof of Lemma \ref{lem-korn} fails in the 3-dimensional case.
\end{remark}

\qed

\subsection{Other useful lemmas}
The following Brouwer's fixed point theorem is crucial to establish the existence. See \cite{Lions1969} or \cite[Lemma IX.3.1]{Galdi2011}.
\begin{lemma}\label{FUNC}
	Let $P$ be a continuous operator which maps $\mathbb{R}^N$ into itself, such that for some $\rho>0$
	\[
	P({\xi})\cdot{\xi}\geq 0 \quad \text { for all } {\xi} \in\mathbb{R}^N \text { with }\,|{\xi}|=\rho.
	\]
	Then there exists ${\xi}_{0} \in\mathbb{R}^N$ with $|{\xi}_{0}| \leq \rho$ such that ${P}({\xi}_{0})=0$.
\end{lemma}

\qed

The following asymptotic estimate of a function that satisfies an ordinary differential inequality will be useful in our further proof. To the best of the authors' knowledge, it was originally derived by Ladyzhenskaya-Solonnikov in \cite{Lady-Sol1980}. We also refer readers to \cite[Lemma 2.7]{LPY:2022ARXIV} for a proof written in a relatively recent format.
\begin{lemma}\label{LEM2.3}
Let $Y(\zeta)\nequiv 0$ be a nondecreasing nonnegative differentiable function satisfying
\[
Y(\zeta)\leq\Psi(Y'(\zeta)),\q\forall\zeta>0.
\]
Here $\Psi:\,[0,\infty)\to[0,\infty)$ is a monotonically increasing function with $\Psi(0)=0$ and there exists $C,\,\tau_1>0$, $m>1$, such that
\[
\Psi(\tau)\leq C\tau^m,\q\forall\tau>\tau_1.
\]
Then
\[
\liminf_{\zeta\to+\infty}\zeta^{-\f{m}{m-1}}Y(\zeta)>0.
\]
\end{lemma}

\qed

The following two lemmas are essential in creating the pressure field for a weak solution to the Navier-Stokes equations. The first one is a special case of \cite[Theorem 17]{DeRham1960} by De Rham. See also \cite[Proposition 1.1]{Temam1984}.
\begin{lemma}\label{DeRham}
For a given open set $\O\subset\mathbb{R}^2$, let $\bl{\mathcal{F}}$ be a distribution in $\left(C_c^\i(\O)\right)'$ which satisfies:
\[
\langle \bl{\mathcal{F}},\bl{\phi}\rangle=0,\q\text{for all}\q \bl{\phi}\in\{\bl{g}\in C_c^\i(\O;\mathbb{R}^2):\,\text{div }\bl{g}=0\}.
\]
Then there exists a distribution $q\in\left(C_c^\i(\O;\mathbb{R})\right)'$ such that
\[
\bl{\mathcal{F}}=\na q.
\]
\end{lemma}

\qed

The second one states the regularity of the aforementioned field $q$:

\begin{lemma}[See \cite{Temam1984}, Proposition 1.2]\label{LEM312}
Let $\Omega$ be a bounded Lipschitz open set in $\mathbb{R}^{2}$. If a distribution $q$ has all its first derivatives $\p_{x_i} q$, $1\leq i \leq2$, in $H^{-1}(\Omega)$, then $q\in L^{2}(\Omega)$ and
\be\label{EEEE0}
\left\|q-\bar{q}_\O\right\|_{L^{2}(\Omega)}\leq C_\O\|\na q\|_{H^{-1}(\Omega)},
\ee
where $\bar{q}_\O=\f{1}{|\O|}\int_{\O}qdx$. Moreover, if $\Omega$ is any Lipschitz open set in $\mathbb{R}^{2}$, then $q\in L_{\mathrm{loc}}^{2}(\overline{\Omega})$.
\end{lemma}

\qed

Finally, we state the following lemma, which shows the existence of the solution to problem $\nabla\cdot \bl{V}=f$ in a truncated regular stripe.
\begin{lemma}\label{LEM2.1}
For a boxed domain $S:= [a,b]\times [c,d]$,  if $f\in L^2(S)$ with $\int_S fdx=0$, then there exists a vector valued function $\bl{V}:\,S\to\mathbb{R}^2$ belongs to $H^1_0(S)$ such that
\be\label{LEM2.11}
\nabla\cdot \bl{V}=f,\q\text{and}\q\|\nabla \bl{V}\|_{L^2(S)}\leq C\|f\|_{L^2(S)}.
\ee
Here $C>0$ is an absolute constant.
\end{lemma}

See \cite{Bme1, Bme2}, also \cite[Chapter III]{Galdi2011} for detailed proof of this lemma.

\qed

\section{Existence and uniqueness of the weak solution}\label{SEU}

\subsection{Construction of the flux carrier}\label{SEC32}

\q\ In this subsection, we are devoted to the construction of a flux carrier $\bl{a}$, which is divergence free, satisfying the Navier-slip boundary condition \eqref{NBC}, and connects two Poiseuille flows in $\mS_L$ and $\mS_R$ smoothly. Meanwhile, the vector $\bl{a}$ will satisfy the following:

\begin{proposition}\label{Prop}
There exists a smooth vector field $\bl{a}(x)$ which enjoys the following properties
\begin{itemize}
\setlength{\itemsep}{-2 pt}
\item[(i).] $\bl{a}\in C^\i(\overline{\mS})$, and $\na\cdot \bl{a}=0$ in $\mS$;

\item[(ii).] $2(\mathbb{S}\bl{a}\cdot\bl{n})_{\mathrm{tan}}+\al \bl{a}_{\mathrm{tan}}=0$, and $\bl{a}\cdot\bl{n}=0$ on $\p\mS$;

\item[(iii).] For a fixed $\ve\in (0,1)$ ,
\be\label{EADF}
\bl{a}=\left\{
\begin{aligned}
&\bl{P}^L_{\Phi}(y)\q\text{ in }\q\mS\cap \{y_1\leq -e^{\f{2}{\ve}}\},\\
&\bl{P}^R_{\Phi}(x)\q\text{ in }\q\mS\cap\{x_1\geq e^{\f{2}{\ve}}\}.
\end{aligned}\right.
\ee
Moreover, for any vector filed $\bl{v}\in \mathcal{V}$ with
\be\label{FUNCVV}
\mathcal{V}:=\left\{\bl{v}\in H^1(\mS):\,\operatorname{div}\bl{v}=0,\ (\bl{v}\cdot\bl{n})\big|_{\p\mS}=0\right\},
\ee
there exists a constant $C$, independent of $\ve$ and $\al$, such that
\be
\lt|\int_{\mS} \bl{v}\cdot \na\bl{a}\cdot \bl{v} dx\rt|\leq C_\mS\Phi\lt(\ve+\f{\al}{1+\al}\rt)\|\na \bl{v}\|^2_{L^2(\mS)}. \label{estisqu}
\ee
\end{itemize}
\end{proposition}

\qed

The following lemma is useful in the construction of $\bl{a}$.

\begin{lemma}
There exists a smooth non-decreasing function $\sigma_\ve: [0,\dl)\to[-\Phi,0]$, where $0<\ve<<1$, such that
\[
\sigma_\ve(t)=\left\{
\begin{aligned}
0,&\q\text{for}\q t\geq \ve;\\
-\Phi,&\q\text{for}\q 0\leq t\leq \e.
\end{aligned}
\right.
\]
Here $\e:=\ve e^{-1/\ve}/3$, and
\[
\dl:=\mathrm{dist}(\p\mS_-,\p\mS_{ob}\cup\p\mS_+)=\inf\left\{|x-z|:\,x\in\p\mS_-,\,z\in\p\mS_{ob}\cup\p\mS_+\right\}>2\ve.
\]
Meanwhile, when $t\in\lt[0,\ve\rt]$, there exists constant $C$ such that
\be\label{VEEST}
\left\{\begin{aligned}
&0\leq\sigma_\ve'(t)\leq \min\left\{\Phi e^{1/\ve},\f{2\Phi\ve}{t}\right\},\\[1mm]
&\left|\sigma_\ve^{(k)}(t)\right|\leq C\Phi e^{1/\ve}(\ve^{-1}e^{1/\ve})^{k-1},\q\text{for }k=2,3.\\
\end{aligned}\right.
\ee
\end{lemma}
\begin{proof}
We start with the piecewise smooth function
\be\label{tauve}
\tau_\ve(t)=\left\{
\begin{aligned}
0,&\q\text{for}\q t\geq\ve;\\
\f{\ve}{t},&\q\text{for}\q \ve e^{-1/\ve}<t<\ve;\\
0,&\q\text{for}\q 0\leq t\leq \ve e^{-1/\ve}.
\end{aligned}
\right.
\ee
Then denoting $\varsigma$ the classical mollifier with radius equals $\e$, the function $\sigma_\ve$ is given by:
\be\label{PRESIG}
\sigma_\ve(t):=\left\{
\begin{array}{ll}
-\Phi,&\q\text{for}\q 0\leq t<\e;\\[2mm]
-\Phi+\tilde{C}\Phi\int_0^t(\varsigma*\tau_\ve)(s)ds,&\q\text{for}\q \e<t<\dl-\e;\\[2mm]
0,&\q\text{for}\q \dl-\e\leq t<\dl.\\
\end{array}\right.
\ee
Here $\tilde{C}>0$ is chosen such that
\[
\tilde{C}\int_0^{\dl-\e}(\varsigma*\tau_\ve)(s)ds=1.
\]
Noting that $\tilde{C}$ must be sufficiently close to $1$, since $\int_0^\dl\tau_\ve(s)ds=1$ by \eqref{tauve}. Finally, \eqref{VEEST}$_1$ follows directly from \eqref{tauve} and \eqref{PRESIG}, while the validity of \eqref{VEEST}$_2$ follows that
\[
\left\|\sigma_\ve^{(k)}\right\|_{L^\i}=\tilde{C}\Phi\left\|\varsigma^{(k-1)}*\tau_\ve\right\|_{L^\i}\leq C\Phi\lt\|\varsigma^{(k-1)}\rt\|_{L^1}\left\|\tau_\ve\right\|_{L^\i}\leq C\Phi e^{1/\ve}(\ve^{-1}e^{1/\ve})^{k-1},\q\text{for }k=2,3.
\]
\end{proof}

{\noindent\bf Proof of Proposition \ref{Prop} : }Given $\ve<\f{\dl}{2}$, we define
\be\label{Cons1}
\bl{a}=\sigma'_\ve(t)\bl{e_s},\q\text{in}\q\mS_{0},
\ee
where $\bl{e_s}$ is defined around \eqref{ESET}, while
\be\label{Cons}
\bl{a}=\left\{
\begin{array}{l} \left[\sigma'_\ve(x_2)(1-\eta(x_1))+\eta(x_1)P^R_\Phi(x_2)\right]\bl{e_1}-\left(\eta'(x_1)\int_0^{x_2}\left(P^R_\Phi(\xi)-\sigma'_\ve(\xi)\right)d\xi\right)\bl{e_2},\\
\hskip 13.5cm\text{in}\q\mS_R;\\[3mm] \left[\sigma'_\ve(y_2)(1-\eta(-y_1))+\eta(-y_1)P^L_\Phi(y_2)\right]\bl{e_1}'+\left(\eta'(-y_1)\int_{0}^{y_2}\left(P^L_\Phi(\xi)-\sigma'_\ve(\xi)\right)d\xi\right)\bl{e_2}',\\
\hskip 13.5cm\text{in}\q\mS_L.\\
\end{array}
\right.
\ee
Here $\eta=\eta(s)$ be the smooth cut-off functions such that
\be\label{ETA}
\eta(s)=\left\{
\begin{aligned}
&1,\q\text{for}\q s>e^{2/\ve};\\
&0,\q\text{for}\q s<0.\\
\end{aligned}
\right.
\ee
and $\eta$ satisfies
\[
|\eta'|\leq 2e^{-2/\ve}, \q\text{and}\q|\eta''|\leq 4e^{-4/\ve}.
\]
$P^L_\Phi$ and $P^R_\Phi$, which are given in \eqref{PF}, are $\bl{e_1}$-component and $\bl{e_1}'$-component of Poiseuille flows in pipes $\mS_L$ and $\mS_R$, respectively.

Using \eqref{ESET} and \eqref{DIV-CURL}$_1$, the flux carrier $\bl{a}$ constructed in \eqref{Cons} is smooth and divergence-free. Meanwhile, since $\sigma_\ve(t)=0$ near $t=0$, one has $\bl{a}$ vanishes near $\p\mS\cap\p\mS_0$. This indicates $\bl{a}$ satisfies the homogeneous Navier-slip boundary condition on $\p\mS\cap\p\mS_0$.

Now we go to verify that $\bl{a}$ meets the Navier-slip boundary condition on $\p\mS\cap(\mS_L\cup\mS_R)$. Owing to cases in $\eqref{Cons}_{1,2}$ are similar, we only consider $\eqref{Cons}_{1}$ for simplicity. Since in this part, $\p\mS_R$ is straight, direct calculation of the Navier-slip boundary condition is to check
\[
\left\{
\begin{aligned}
&-\p_{x_2}a_1(x_1,0)+\al a_1(x_1,0)=\p_{x_2}a_1(x_1,1)+\al a_1(x_1,1)=0;\\[2mm]
&a_2(x_1,0)=a_2(x_1,1)=0,\\
\end{aligned}
\right.\q\q\forall x_1\in(0,\i).
\]
This could be done by the definition of $P^R_\Phi$ in \eqref{PF}, the construction of $\sigma_\ve$ above, and direct calculations.

Then items (i) and item (ii) in Proposition \ref{Prop} is proven and also it is easy to check that \eqref{EADF} stands due to the choice of the cutoff function $\eta(x_1)$. Now it remains to derive \eqref{estisqu}, we define
\[
V:=\int_{\mS}\bl{v}\cdot\nabla \bl{a}\cdot \bl{v} dx=\int_{\mS_L\cup\mS_0\cup\mS_R} \bl{v}\cdot\nabla \bl{a}\cdot \bl{v}dx.
\]

{\noindent\bf Estimates of the $\mS_0$-part integration:}\\[1mm]

In the distorted part $\mS_0$, we denote that
\[
\bl{v}=v_s(s,t)\bl{e_s}+v_t(s,t)\bl{e_t},
\]
where the coordinates $(s,t)$ and vectors $\bl{e_s}$, $\bl{e_t}$ are defined in Section \ref{PRE}. Noting that
\[\bl{e_t}\cdot\na=\p_t\,,\quad \bl{e_s}\cdot\na=\gamma(s,t)\p_s\,,\quad \f{d\bl{e_s}}{ds}=\f{\kappa(s,t)}{\gamma(s,t)}\bl{e_t}\,,\]
in $\mS_0$, one derives
\[
\begin{split}
\bl{v}\cdot\nabla \bl{a}&=\gamma(s,t) v_s(s,t)\f{\p}{\p s}(\sigma_\ve'(t)\bl{e_s})+v_t(s,t)\f{\p}{\p t}(\sigma_\ve'(t)\bl{e_s})\\
&=\kappa(s,t) v_s(s,t)\sigma_\ve'(t)\bl{e_t}+v_t(s,t)\sigma_\ve''(t)\bl{e_s}\,.
\end{split}
\]
Hence, we have
\[
\bl{v}\cdot\nabla \bl{a}\cdot \bl{v}=\left(\sigma_\ve''(t)+\kappa(s,t)\sigma_\ve'(t)\right)v_s(s,t)v_t(s,t),\q\text{in}\q\mS_0.
\]
Recalling \eqref{para}, we deduce
\[
\begin{split}
	\int_{\mS_0}\bl{v}\cdot\nabla \bl{a}\cdot \bl{v}dx=&\int_0^\dl\int_{-s_0}^{s_0}\left(\sigma_\ve''(t)+\kappa(s,t)\sigma_\ve'(t)\right)v_s(s,t)v_t(s,t)\f{1}{\gamma(s,t)}ds dt\\[1mm]
	=&\int_0^\dl\int_{-s_0}^{s_0}\sigma_\ve''(t)v_s(s,t)v_t(s,t)\f{1}{\gamma(s,t)}dsdt+\int_0^\dl\int_{-s_0}^{s_0}\f{\kappa(s,t)}{\gamma(s,t)}\sigma_\ve'(t)v_s(s,t)v_t(s,t)dsdt.
\end{split}
\]
Noting that $\sigma'_\ve(t)$ vanishes near $\p\mS\cap\p\mS_0$, integration by parts for the first term of the right hand side of the above equality on $t$ indicate that
\be\label{s0est}
\begin{split}
\int_{\mS_0}\bl{v}\cdot\nabla \bl{a}\cdot \bl{v}dx=&-\int_0^\dl\int_{-s_0}^{s_0}\sigma_\ve'(t)\p_t\left(\f{v_s}{\gamma}\right)v_tdsdt-\int_0^\dl\int_{-s_0}^{s_0}\sigma_\ve'(t)v_s\p_tv_t\f{1}{\gamma(t,s)}dsdt\\
&+\int_0^\dl\int_{-s_0}^{s_0}\f{\kappa(s,t)}{\gamma(s,t)}\sigma_\ve'(t)v_s(s,t)v_t(s,t)dsdt.
\end{split}
\ee
Recalling \eqref{DIV-CURL}$_1$, the divergence-free property of $\bl{v}$ in the curvilinear coordinates follows:
\[
\gamma\p_sv_s+\p_tv_t-\kappa v_t=0.
\]
Then inserting the divergence-free property into \eqref{s0est}, we obtain that
\be\label{s1est}
\begin{split}
\int_{\mS_0}\bl{v}\cdot\nabla \bl{a}\cdot \bl{v}dx&=-\int_0^\dl\int_{-s_0}^{s_0}\sigma_\ve'(t)\p_t\left(\f{v_s}{\gamma}\right)v_tdsdt+\int_0^\dl\int_{-s_0}^{s_0}\sigma_\ve'(t)v_s\p_sv_sdsdt\\
&:=V_1+V_2.
\end{split}
\ee
First, noting that $\gamma(s,t)>\gamma_0>0$ is smooth, and $\sigma_\ve'(t)$, which is supported on $[0,\ve]$, satisfies
\[
|\sigma_\ve'(t)|\leq\f{2\Phi\ve}{t},\q\forall t\in[0,\ve],
\]
one bounds $V_1$ by using the Cauchy-Schwarz inequality and the Poincar\'e inequality in Lemma \ref{TORPOIN}
\be\label{V210}
\begin{split}
|V_{1}|\leq& C \Phi\ve\left(\int_0^\dl\int_{-s_0}^{s_0}|\na \bl{v}|^2dsdt\right)^{1/2}\un{\left(\int_0^\dl\int_{-s_0}^{s_0}\f{|v_t|^2}{t^2}dsdt\right)^{1/2}}_{V_{11}}.
\end{split}
\ee
Due to $v_t=0$ on the $t=0$, the part $V_{11}$ can be estimated by the one dimensional Hardy inequality. In fact
\[
\begin{split}
\int_0^\dl\f{|v_t(s,t)|^2}{t^2}dt&=-\int_0^\dl|v_t(s,t)|^2\left(\f{1}{t}\right)'dt\\
&=-\f{|v_t(s,\dl)|^2}{\dl}\int_0^\dl\f{|v_t(s,t)|^2}{t^2}dt+2\int_0^\dl \f{v_t(s,t)}{t}\p_tv_t(s,t)dt\\
&\leq 2\left(\int_0^\dl\f{|v_t(s,t)|^2}{t^2}dt\right)^{1/2}\left(\int_0^\dl|\p_tv_t(s,t)|^2dt\right)^{1/2},
\end{split}
\]
which indicates
\be\label{V211}
\int_0^\dl\f{|v_t(s,t)|^2}{t^2}dt\leq 4\int_0^\dl|\p_tv_t(s,t)|^2dt.
\ee
Thus one concludes
\[
|V_{1}|\leq C\Phi\ve\|\nabla \bl{v}\|_{L^2(\mS_0)}^2
\]
by combining \eqref{V210} and \eqref{V211}. For $V_{2}$ in \eqref{s1est}, it follows that
\be\label{V221E}
\begin{aligned}
V_{2}&=\f{1}{2}\int_0^\dl\int_{-s_0}^{s_0}\sigma_\ve'(t)\p_s(v_s)^2dsdt\\
	&=\f{1}{2}\int_{0}^1\sigma_\ve'(x_2)(v_1)^2(0,x_2)dx_2-\f{1}{2}\int_{0}^{c_0}\sigma_\ve'(y_2)(v_1)^2(0,y_2)dy_2.
\end{aligned}
\ee
The second equality above is established due to the Newton-Leibniz formula and fact that the curvilinear coordinates  $(s,t)$ turns to be Euclidean in $\mS\backslash\mS_0$.\\[1mm]

{\noindent\bf Estimates in $\mS_L$ and $\mS_R$.}\\[1mm]

The cases in subsets $\mS_L$ and $\mS_R$ are similar, thus we only discuss the latter one for simplicity. At the beginning, we denote that
\[
\mS_R:=\mS_{R1}\cup\mS_{R2},
\]
where
\[
\left\{
\begin{aligned}
&\mS_{R1}=\mS\cap\{x_1\in[0,\,e^{2/\ve}]\};\\
&\mS_{R2}=\mS\cap\{x_1\in(e^{2/\ve},\i)\}.
\end{aligned}
\right.
\]
Direct calculation shows
\[
\begin{aligned}
	\int_{\mS_R}\bl{v}\cdot\nabla \bl{a}\cdot \bl{v}dx=&\int_{\mS_{R1}}\eta'(x_1)\left(P^R_\Phi(x_2)-\sigma'_\ve(x_2)\right)\left((v_1)^2-(v_2)^2\right)dx\\
	&+\int_{\mS_{R1}}\eta''(x_1)\left(\int_0^{x_2}\left(\sigma'_\ve(\xi)-P^R_\Phi(\xi)\right)d\xi\right)v_1v_2dx\\
	&+\int_{\mS_{R1}}\sigma_\ve''(x_2)\left(1-\eta(x_1)\right)v_1v_2dx+\int_{\mS_{R2}}\eta(x_1)(P_{\Phi}^R)'(x_2)v_1v_2dx\\
:=&J_1+J_2+J_3+J_4.
\end{aligned}
\]
Here, noticing that $|\eta'|\leq 2e^{-2/\ve}$, $|\eta''|\leq 4e^{-4\ve}$ and
\[
|\sigma_\ve'(t)|\leq \Phi e^{1/\ve}\q\forall t\in[0,\dl)
\]
which follows from \eqref{VEEST}, one concludes that
\[
|J_1|+|J_2|\leq C\Phi e^{-1/\ve}\int_{\mS_{R1}}|\bl{v}|^2dx\leq C\Phi\ve\|\nabla\bl{v}\|_{L^2(\mS)}^2
\]
by applying the Cauchy-Schwarz inequality and the Poincar\'e inequality in Corollary \ref{CorPoin}. Moreover, adopting integrating by parts, one deduces
\[
\begin{split}
J_3=&-\int_{\mS_{R1}}\sigma_\ve'(x_2)\left(1-\eta(x_1)\right)\p_{x_2}v_1v_2dx-\int_{\mS_{R1}}\sigma_\ve'(x_2)\left(1-\eta(x_1)\right)v_1\p_{x_2}v_2dx\\
:=&J_{31}+J_{32}.
\end{split}
\]
Via an analogous route as we go through for $V_{1}$ in \eqref{s1est} above, one deduces
\[
|J_{31}|\leq C \Phi\ve\|\nabla\bl{v}\|_{L^2(\mS)}^2.
\]
For the term $J_{32}$, applying the divergence-free property of $\bl{v}$ and using integration by parts, one arrives
\[
\begin{aligned}
J_{32}=&\f{1}{2}\int_{\mS_{R1}}\sigma_\ve'(x_2)\left(1-\eta(x_1)\right)\p_{x_1}(v_1)^2dx\\ =&\f{1}{2}\int_{\mS_{R1}}\sigma_\ve'(x_2)\eta'(x_1)(v_1)^2dx-\f{1}{2}\int_{0}^1\sigma_\ve'(x_2)(v_1)^2(0,x_2)dx_2\\
:=&J_{321}+J_{322}.
\end{aligned}
\]
Noting that $J_{321}$ can be estimated in the same way as we do on $J_1$ and $J_2$, that is
\[
|J_{321}|\leq C \Phi\ve\|\nabla\bl{v}\|_{L^2(\mS)}^2.
\]
Due to $J_{322}$ being cancelled out with the first term in \eqref{V221E}$_2$, it remains only to estimate $J_4$. Recall the $L^\i$ bound of $(P_\Phi^R)'$ in \eqref{EHP}, one concludes that
\[
J_4\leq C\f{\al\Phi}{1+\al}\|\na \bl{v}^N\|_{L^2(\mS)}^2.
\]
Collecting the above estimates and cancellations, we derive that
\[
|V|=\left|\int_{\mS}\bl{v}\cdot\nabla \bl{a}\cdot \bl{v}dx\right|\leq C\Phi \left(\ve+\f{\al}{1+\al}\right)\|\na \bl{v}\|_{L^2(\mS)}^2,
\]
which concludes \eqref{estisqu}. This completes the proof of Proposition \ref{Prop}.

\qed

\subsection{Existence of the weak solution}\label{SEC3627}
We will look for a solution to \eqref{NS}-\eqref{GL3} of the form
\be\label{EEPPSS}
\bl{u}=\bl{v}+\bl{a}.
\ee
Thus, our problem turns to the following equivalent form:
\begin{problem}\label{PP2}
Find $(\bl{v},p)$ such that
\[
	\left\{
	\begin{aligned}
		&\bl{v}\cdot\nabla \bl{v}+\bl{a}\cdot\nabla \bl{v}+\bl{v}\cdot\nabla \bl{a}+\na p-\Dl \bl{v}=\Dl \bl{a}-\bl{a}\cdot\nabla \bl{a},\\
		&\na\cdot \bl{v}=0,\\
	\end{aligned}
	\right.\q\text{in }\q\mS,
\]
	subject to the Navier-slip boundary condition
	\be\label{BDR}
	\left\{
	\begin{aligned}
		&2(\mathbb{S}\bl{v}\cdot\bl{n})_{\mathrm{tan}}+\al \bl{v}_{\mathrm{tan}}=0,\\
		&\bl{v}\cdot\bl{n}=0,\\
	\end{aligned}
	\right.\q\text{on }\q\p\mS,
	\ee
	with the asymptotic behavior as $|x|\to\i$
	\[
	\bl{v}(x)\to \bl{0},\q\text{as}\q |x|\to \i.
	\]
\end{problem}

\qed

From the weak formulation \eqref{weaksd}, we have that $\bl{v}$ satisfies the following weak formulation:
\begin{definition}\label{defws}
	Let $\bl{a}$ be a smooth vector satisfying the properties stated in the above. We say that $\bl{v}\in \mathcal{H}_\sigma(\mS)$ is a weak solution of Problem \ref{PP2} if
	\be\label{EQU111}
	\begin{split}
		&2\int_{\mS}\mathbb{S}\bl{v}:\mathbb{S}\bl{\varphi} dx+\al\int_{\p\mS}\bl{v}_{\mathrm{tan}}\cdot\bl{\varphi}_{\mathrm{tan}}dS+\int_{\mS}\bl{v}\cdot\nabla \bl{v} \cdot \bl{\varphi} dx+\int_{\mS}\bl{v}\cdot\nabla \bl{a}\cdot\bl{\varphi} dx\\
		+&\int_{\mS}\bl{a}\cdot\nabla \bl{v}\cdot\bl{\varphi} dx=\int_{\mS}\big(\Dl \bl{a}-\bl{a}\cdot\nabla \bl{a}\big)\cdot\bl{\varphi} dx
	\end{split}
	\ee
	holds for any vector-valued function $\bl{\varphi}\in \mathcal{H}_\sigma(\mS)$.
\end{definition}

\qed

Now we state our main result of this part.
\begin{theorem}\label{THM3.8}
	There is a constant $\Phi_0>0$ depending on the curvature of $\p\mS$ such that if $\f{\al\Phi}{1+\al}<\Phi_0$, then Problem \ref{PP2} admits at least one weak solution $(\bl{v},p)\in \mathcal{H}_\sigma(\mS)\times L^2_{\mathrm{loc}}(\overline{\mS})$, with
\be\label{EOFV627}
	\|\bl{v}\|_{H^1(\mS)}\leq C(\|\bl{a}\cdot\na \bl{a}\|_{L^2(\mS_{e^{C\Phi}})}+\|\Dl\bl{a}\|_{L^2(\mS_{e^{C\Phi}})})\leq \Phi e^{C\Phi}\,.
\ee
\end{theorem}

\qed
\begin{proof}
Set
\[
\bl{X}:=C^\infty_{\sigma,c}(\overline{\mS};\,\mathbb{R}^2)=\big\{\bl{\varphi}\in C^\infty_c(\overline{\mS}\,;\,\mathbb{R}^2):\, \na\cdot \bl{\varphi}=0,\ \bl{\varphi}\cdot\bl{n}\big|_{\p\mS}=0\big\},
\]
and $\{\bl{\varphi}_k\}_{k=1}^\i\subset\bl{X}$ be an unit orthonormal basis of $\mathcal{H}_\sigma(\mS)$, that is:
\[
\langle \bl{\varphi}_i,\bl{\varphi}_j\rangle_{H^1(\mS)}=\begin{cases}
	1,&\q\text{if }\q i=j;\\
	0,&\q\text{if }\q i\neq j,\\
\end{cases}
\]
$\forall i,j\in\mathbb{N}$. We look for an approximation of $\bl{v}$ of the form
\[
\bl{v}^N(x)=\sum_{i=1}^Nc_i^N\bl{\varphi}_i(x).
\]
Testing the weak formulation \eqref{EQU111} by $\bl{\varphi}_i$, with $i=1,2,...,N$, one has
\[
\begin{split}
	&2\sum_{i=1}^Nc_i^N\int_{\mS}\mathbb{S}\bl{\varphi}_i:\mathbb{S}\bl{\varphi}_j dx+\al\sum_{i=1}^N c_i^N\int_{\p\mS} (\bl{\varphi}_i)_{\mathrm{tan}}(\bl{\varphi}_j)_{\mathrm{tan}}dS+\sum_{i,k=1}^Nc_i^Nc_k^N\int_{\mS}\bl{\varphi}_i\cdot\nabla \bl{\varphi}_k\cdot \bl{\varphi}_jdx\\
	+&\sum_{i=1}^N\int_{\mS}\bl{\varphi}_i\cdot\nabla \bl{a}\cdot \bl{\varphi}_jdx+\sum_{i=1}^Nc_i^N\int_{\mS}\bl{a}\cdot\nabla \bl{\varphi}_i\cdot \bl{\varphi}_jdx=\int_{\mS}\big(\Dl \bl{a}-\bl{a}\cdot\nabla \bl{a}\big)\cdot \bl{\varphi}_jdx,\q\forall j=1,2,...,N.
\end{split}
\]
This is a system of nonlinear algebraic equations of $N$-dimensional vector
\[
\bl{c}^N:=(c_1^N,c_2^N,...,c_N^N).
\]
We denote ${P}:\,\mathbb{R}^N\to\mathbb{R}^N$ such that
\bes
\begin{split}
	\big({P}(\bl{c}^N)\big)_j=&2\sum_{i=1}^Nc_i^N\int_{\mS}\mathbb{S}\bl{\varphi}_i:\mathbb{S}\bl{\varphi}_j dx+\al\sum_{i=1}^N c_i^N\int_{\p\mS} (\bl{\varphi}_i)_{\mathrm{tan}}\cdot(\bl{\varphi}_j)_{\mathrm{tan}}dS+\sum_{i,k=1}^Nc_i^Nc_k^N\int_{\mS}\bl{\varphi}_i\cdot\nabla \bl{\varphi}_k\cdot \bl{\varphi}_jdx\\
	&+\sum_{i=1}^N\int_{\mS}\bl{\varphi}_i\cdot\nabla \bl{a}\cdot \bl{\varphi}_jdx+\sum_{i=1}^Nc_i^N\int_{\mS}\bl{a}\cdot\nabla \bl{\varphi}_i\cdot \bl{\varphi}_jdx-\int_{\mS}\big(\Dl \bl{a}-\bl{a}\cdot\nabla \bl{a}\big)\cdot \bl{\varphi}_jdx,\\
	&\hskip 10cm\q\forall j=1,2,...,N.
\end{split}
\ees
It is easy to check that
\[
{P}(\bl{c}^N)\cdot\bl{c}^N=\un{2\int_{\mS}|\mathbb{S} \bl{v}^N|^2dx+\al\int_{\p\mS}|(\bl{v}^N)_{\mathrm{tan}}|^2dS}_{A_1}+\un{\int_{\mS}\left((\bl{v}^N+\bl{a})\cdot\nabla(\bl{v}^N+\bl{a})\right)\cdot \bl{v}^Ndx}_{A_2}-\un{\int_{\mS}\bl{v}^N\cdot\Dl \bl{a}dx}_{A_3}.
\]
By Lemma \ref{lem-korn}, we have
\[
A_1\geq C_0\int_{\mS}|\nabla \bl{v}^N|^2dx.
\]

Next, by using integration by parts, together with the divergence-free property of $\bl{v}^N$ and $\bl{a}$, one knows that
\[
A_2=\un{\int_{\mS}\bl{v}^N\cdot\nabla \bl{a}\cdot \bl{v}^Ndx}_{V}+\un{\int_{\mS}\bl{a}\cdot\nabla \bl{a}\cdot \bl{v}^Ndx}_{K}.
\]
We now focus on the term $V$. Applying \emph{(iii)} in Proposition \ref{Prop} to $\bl{v}^N$, one deduces
\be\label{ESTA21}
|V|\leq C_1\Phi\left(\ve+\f{\al}{1+\al}\right)\|\na\bl{v}^N\|_{L^2(\mS)}^2,
\ee
where the constant $C_1$ is independent with $N$. For the term $K$, since $\bl{a}$ equals to the Poiseuille flow $\bl{P}^L_{\Phi}$ or $\bl{P}^R_{\Phi}$ in $\mS-\mS_{e^{2/\ve}}$, we have $\bl{a}\cdot\nabla \bl{a}\equiv 0$ in $\mS-\mS_{e^{2/\ve}}$.  Using the Cauchy-Schwarz inequality and the Poincar\'e inequality, one arrives at
\be\label{ESTA22}
\begin{split}
|K|=\left|\int_{\mS}\bl{a}\cdot\nabla \bl{a}\cdot \bl{v}^Ndx\right|
\leq C\|\bl{a}\cdot\na \bl{a}\|_{L^2(\mS_{e^{2/\ve}})}\|\na\bl{v}^N\|_{L^2(\mS)}\,.
\end{split}
\ee
Thus, by combining \eqref{ESTA21} and \eqref{ESTA22}, one deduces
\[
|A_2|\leq C_1\Phi\left(\ve+\f{\al}{1+\al}\right)\|\na\bl{v}^N\|_{L^2(\mS)}^2+C\|\bl{a}\cdot\na \bl{a}\|_{L^2(\mS_{e^{2/\ve}})}\|\na\bl{v}^N\|_{L^2(\mS)}.
\]
Finally, by the construction of the Poiseuille flow $\bl{P}_{\Phi}^{R}$ and $\bl{P}_{\Phi}^{L}$, we have
\[
\int_{(-\i,-e^{2/\ve})\times(0,c_0)}\bl{v}^N\cdot\Dl \bl{a}dy=-C_L\int^{-e^{2/\ve}}_{-\i}\int_{0}^{c_0}\bl{v}^N\cdot\bl{e_1}'dy=0
\]
and
\[
\int_{(e^{2/\ve},+\infty)\times(0,1)}\bl{v}^N\cdot\Dl \bl{a}dx=-C_R\int^{\i}_{e^{2/\ve}}\int_{0}^{1}\bl{v}^N\cdot\bl{e_1}dx=0\,.
\]
Thus, by the Cauchy-Schwarz inequality and the Poincar\'e inequality, we deduce that
\[
|A_3|=\left|\int_{\mS_{e^{2/\ve}}}\bl{v}^N\cdot\Dl\bl{a}dx\right|\leq C\|\Dl\bl{a}\|_{L^2(\mS_{e^{2/\ve}})}\|\na\bl{v}^N\|_{L^2(\mS)}.
\]

Substituting the above estimates for $A_1$--$A_3$, and choosing $\ve>0$ being sufficiently small such that $C_1\ve\Phi<\f{C_0}{2}$, one derives
\[
{P}(\bl{c}^N)\cdot\bl{c}^N\geq\|\bl{v}^N\|_{H^1(\mS)}\left(\left(\f{C_0}{2}-C_1\f{\al\Phi}{1+\al}\right)\|\bl{v}^N\|_{H^1(\mS)}-C(\|\bl{a}\cdot\na \bl{a}\|_{L^2(\mS_{e^{2/\ve}})}+\|\Dl\bl{a}\|_{L^2(\mS_{e^{2/\ve}})})\right),
\]
which guarantees
\[
{P}(\bl{c}^N)\cdot\bl{c}^N\geq0,
\]
provided
\[
\f{\al\Phi}{1+\al}< \Phi_0:=\f12C_1^{-1}C_0\q\text{and}\q|\bl{c}^N|=\|\bl{v}^N\|_{H^1(\mS)}\geq\f{C\left(\|\bl{a}\cdot\na \bl{a}\|_{L^2(\mS_{e^{2/\ve}})}+\|\Dl\bl{a}\|_{L^2(\mS_{e^{2/\ve}})}\right)}{C_0/2-C_1\al\Phi/(1+\al)}:=\rho.
\]
Using Lemma \ref{FUNC}, there exists
\be\label{ubound}
(\bl{v}^N)^*\in\text{span}\left\{\bl{\varphi}_1,\bl{\varphi}_2,...,\bl{\varphi}_N\right\},\q\text{and}\q\|(\bl{v}^N)^*\|_{H^1(\mS)}\leq\f{C\left(\|\bl{a}\cdot\na \bl{a}\|_{L^2(\mS_{e^{2/\ve}})}+\|\Dl\bl{a}\|_{L^2(\mS_{e^{2/\ve}})}\right)}{C_0/2-C_1\al\Phi/(1+\al)},
\ee
such that
\begin{equation}\label{app-N}
	\begin{split}
		&2\int_{\mS}\mathbb{S}(\bl{v}^N)^*:\mathbb{S}\bl{\phi}_Ndx+\al\int_{\p\mS}(\bl{v}^N)^*_{\mathrm{tan}}\cdot(\bl{\phi}_N)_{\mathrm{tan}}dS+\int_{\mS}(\bl{v}^N)^*\cdot\nabla (\bl{v}^N)^*\cdot \bl{\phi}_Ndx+\int_{\mS}(\bl{v}^N)^*\cdot\nabla \bl{a}\cdot\bl{\phi}_Ndx\\
		+&\int_{\mS}\bl{a}\cdot\nabla(\bl{v}^N)^*\cdot\bl{\phi}_N dx=\int_{\mS}\big(\Dl \bl{a}-\bl{a}\cdot\nabla \bl{a}\big)\cdot\bl{\phi}_Ndx,\q\forall\bl{\phi}_N\in\text{span}\left\{\bl{\varphi}_1,\bl{\varphi}_2,...,\bl{\varphi}_N\right\}.
	\end{split}
\end{equation}

The above bound \eqref{ubound} and Rellich-Kondrachov embedding theorem imply the existence of a field $\bl{v}\in \mH_{\sigma}(\mS)$ and a subsequence, which we will always denote by $(\bl{v}^N)^*$, such that
\begin{equation*}
	(\bl{v}^N)^*\to \bl{v}\quad \text{weakly in $\mH_{\sigma}(\mS)$}
\end{equation*}
and
\begin{equation*}
	(\bl{v}^N)^*\to \bl{v}\quad \text{strongly in $L^2(\mS')$, for all bounded $\mS'\subset\mS$}\,.
\end{equation*}
By passing to the limit in \eqref{app-N}, one obtains
\be\label{V}
\begin{split}
	&2\int_{\mS}\mathbb{S}\bl{v}:\mathbb{S}\bl{\varphi} dx+\al\int_{\p\mS}\bl{v}_{\mathrm{tan}}\cdot\bl{\varphi}_{\mathrm{tan}}dS+\int_{\mS}\bl{v}\cdot\nabla \bl{v}\cdot\bl{\varphi} dx+\int_{\mS}\bl{v}\cdot\nabla \bl{a}\cdot\bl{\varphi} dx\\
	+&\int_{\mS}\bl{a}\cdot\nabla \bl{v}\cdot\bl{\varphi} dx=\int_{\mS}\big(\Dl \bl{a}-\bl{a}\cdot\nabla \bl{a}\big)\cdot\bl{\varphi} dx,\q\text{for any}\q \bl{\varphi}\in \mH_\sigma(\mS).
\end{split}
\ee
It follows from \eqref{ubound} and the Fatou lemma for weakly convergent sequences that
\be\label{EEEEST1}
\|\bl{v}\|_{H^1(\mS)}\leq C\left(\|\bl{a}\cdot\na \bl{a}\|_{L^2(\mS_{e^{2/\ve}})}+\|\Dl\bl{a}\|_{L^2(\mS_{e^{2/\ve}})}\right)\,.
\ee
Now it remains to verify \eqref{EOFV627}. From the construction of $\bl{a}$ in \eqref{Cons1}--\eqref{Cons} and the estimate of $\sigma_\ve$ in \eqref{VEEST}, we have
\[
\left|\na^k\bl{a}\right|\leq C\Phi e^{\f1{\ve}}(\ve^{-1}e^{\f1{\ve}})^k,\q\text{for }k=0,1,2.
\]
According to the construction of $\bl{v}$ given before, it is legal to choose $\ve=\min\left\{\f{C_0}{4C_1\Phi},\f{\delta}{2}\right\}$. This indicates that
\[
\|\bl{a}\cdot\na \bl{a}\|_{L^2(\mS_{e^{2/\ve}})}+\|\Dl\bl{a}\|_{L^2(\mS_{e^{2/\ve}})}\leq Ce^{1/\ve}\left(\Phi^2\ve^{-1}e^{3/\ve}+\Phi\ve^{-2}e^{3/\ve}\right)\leq C\Phi\left(1+\Phi^2e^{C\Phi}\right),
\]
which gives
\[
\|\bl{v}\|_{H^1(\mS)}\leq \Phi e^{C\Phi}\,.
\]

Now we focus on the pressure. Let $\bl{v}$ be a weak solution of \eqref{EQU111} constructed in the above. Using \eqref{V}, one has $\bl{u}=\bl{v}+\bl{a}$ satisfies
\[
\int_{\mS}\nabla \bl{u}\cdot\nabla\bl{\phi}\,dx+\int_{\mS}\bl{u}\cdot\nabla \bl{u}\cdot\bl{\phi}\,dx=0,\q\text{for all}\q \bl{\phi}\in\{\bl{g}\in C_c^\i(\mS;\mathbb{R}^2):\,\text{div }\bl{g}=0\}.
\]
Thus by Lemma \ref{DeRham}, there exists $p\in\left(C_c^\i(\mS;\mathbb{R})\right)'$, such that
\be\label{NS11}
\Dl\bl{u}-\bl{u}\cdot\nabla\bl{u}=\nabla p
\ee
in the sense of distribution. Furthermore, we have that \eqref{NS11} is equivalent to
\be\label{DEFPi}
\text{div}\big(\nabla\bl{v}-\bl{v}\otimes\bl{v}-\bl{a}\otimes\bl{v}-\bl{v}\otimes\bl{a}\big)+\Dl \bl{a}+{C_{R}}{\eta(x_1)}\bl{e_1}+{C_{L}}{\eta(-y_1)}\bl{e_1}'-\bl{a}\cdot\nabla \bl{a}=\na\Pi,
\ee
with
\be\label{DEFPI}
\Pi=p+{C_{R}}{\int_{-\i}^{x_1}\eta(s)ds}-{C_{L}}{\int_{-\i}^{-y_1}\eta(s)ds}\,,
\ee
where $C_L$ and $C_R$ are Poiseuille constants defined in \eqref{POSS1}$_1$ and \eqref{POSS2}$_1$, respectively. By the definition of $\bl{a}$, one has both
\[
\Dl \bl{a}+{C_{R}}{\eta(x_1)}\bl{e_1}+{C_{L}}{\eta(-y_1)}\bl{e_1}'
\]
and $\bl{a}\cdot\nabla\bl{a}$ are smooth and have compact support. Since $\bl{v}\in H^1(\mS)$ and $\bl{a}$ is uniformly bounded, one deduces
\[
\nabla\bl{v}-\bl{v}\otimes\bl{v}-\bl{a}\otimes\bl{v}-\bl{v}\otimes\bl{a}\in L^2(\mS),
\]
directly by the Sobolev embedding and H\"older's inequality. Therefore one concludes the left hand side of \eqref{DEFPi} belongs to $H^{-1}(\mS)$. Then applying Lemma \ref{LEM312}, we have $\Pi\in L^2_{\mathrm{loc}}(\overline{\mS})$, which leads to $p\in L^2_{\mathrm{loc}}(\overline{\mS})$ by \eqref{DEFPI}.
\end{proof}

\subsection{Uniqueness result}
The rest part of this section is devoted to the proof of uniqueness. We will show that the solution $(\bl{u},{p})$ constructed earlier in this section with its flux being $\Phi$ is unique for $\Phi$ being sufficiently small and independent of $\al$.

\subsubsection{Estimate of the pressure}
Below, we give a proposition to show that an integration estimate related to the pressure in the truncated strip $\Upsilon_Z^{+}:=\lt(\mS_Z\backslash\mS_{Z-1}\rt)\cap \{x_1>0\}$ or $\Upsilon_Z^{-}:=\lt(\mS_Z\backslash\mS_{Z-1}\rt)\cap \{y_1<0\}$.

\begin{proposition}\label{P2.4}

Let $(\tilde{\bl{u}},\tilde{p})$ be an alternative weak solution of \eqref{NS} in the strip $\mS$, subject to the Navier-slip boundary condition \eqref{NBC}. If the total flux
\[
\int_{\mS\cap\{x_1=s\}}\tilde{\bl{u}}(s,x_2)\cdot\boldsymbol{e_1}dx_2=\Phi=\int_{\mS\cap\{x_1=s\}}{\bl{u}}(s,x_2)\cdot\boldsymbol{e_1}dx_2,\q\text{for any } s\geq 1,
\]
then the following estimate of $\bl{w}:=\tilde{\bl{u}}-{\bl{u}}$ and the pressure holds
\be\label{ZJZJ}
\left|\int_{\Upsilon^{\pm}_K}(\tilde{p}-p)w_1dx\right|\leq C\left(\|\bl{u}\|_{L^4(\Upsilon^{\pm}_K)}\|\nabla \bl{w}\|^2_{L^2(\Upsilon^{\pm}_K)}+\|\nabla \bl{w}\|_{L^2(\Upsilon^{\pm}_K)}^2+\|\nabla \bl{w}\|_{L^2(\Upsilon^{\pm}_K)}^3\right),\forall K\geq 2,
\ee
where $C>0$ is a constant independent of $K$.
\end{proposition}
\pf We only show \eqref{ZJZJ} on the $\Upsilon^+$ since the rest part is similar. During the proof, we cancel the upper index ``$+$" of the domain for simplicity. Noticing
\[
\int_{\mS\cap\{x_1=s\}}w_1(s,x_2)dx_2\equiv0,\q\forall s\geq 1,
\]
by integrating the above equality for variable $s$ from $K-1$ to $K$,  we deduce that
\[
\int_{\Upsilon_K}w_1 dx=0,\q\forall K\geq 2.
\]
Using Lemma \ref{LEM2.1}, one derives the existence of a vector field $V$ satisfying \eqref{LEM2.11} with $f=w_1$. Applying equation \eqref{NS}$_1$, one arrives
\[
\bali
\int_{\Upsilon_K}(\tilde{p}-{p})w_1dx=&\int_{\Upsilon_K}(\tilde{p}-{p})\na\cdot \bl{V}dx\\
=&-\int_{\Upsilon_K}\nabla(\tilde{p}-{p})\cdot \bl{V}dx=\int_{\Upsilon_K}\left(\bl{w}\cdot\na \bl{w}+\bl{u}\cdot\nabla \bl{w}+\bl{w}\cdot\nabla \bl{u}-\Dl \bl{w}\right)\cdot \bl{V}dx.
\eali
\]
Using integration by parts, one deduces
\[
\int_{\Upsilon_K}(\tilde{p}-p)w_1dx=\sum_{i,j=1}^2\int_{\Upsilon_K}(\p_iw_j-w_iw_j-u_iw_j-u_jw_i)\p_iV_jdx.
\]
By applying H\"older's inequality and \eqref{LEM2.11} in Lemma \ref{LEM2.1}, one deduces that
\be\label{EP1}
\left|\int_{\Upsilon_K}(\tilde{p}-p)w_1dx\right|\leq C\left(\|\na \bl{w}\|_{L^2(\Upsilon_K)}+\|\bl{w}\|_{L^4(\Upsilon_K)}^2+\|\bl{u}\|_{L^4(\Upsilon_K)}\|\bl{w}\|_{L^4(\Upsilon_K)}\right)\|w_1\|_{L^2(\Upsilon_K)}.
\ee
Since $w_1$ has a zero mean value on each cross section $\{x_1=s\}$ for $s\geq 1$ and $w_2$ has zero boundary on in the $x_2$ direction, then Poincar\'e inequality in $x_2$ direction implies that
\be\label{PON3}
\|\bl{w}\|_{L^2(\Upsilon_K)}\leq C  \|\p_{x_2} \bl{w}\|_{L^2(\Upsilon_K)}.
\ee
Substituting \eqref{PON3} in \eqref{EP1}, also noting the Gagliardo-Nirenberg inequality
\[
\|\bl{w}\|_{L^4(\Upsilon_K)}^2\leq C\left(\|\bl{w}\|_{L^2(\Upsilon_K)}\|\na \bl{w}\|_{L^2(\Upsilon_K)}+\|\bl{w}\|_{L^2(\Upsilon_K)}^{2}\right),
\]
one concludes
\[
\left|\int_{\Upsilon_K}(\tilde{p}-p)w_1dx\right|\leq C\left(\|\bl{u}\|_{L^4(\Upsilon_K)}\|\nabla \bl{w}\|^2_{L^2(\Upsilon_K)}+\|\nabla \bl{w}\|_{L^2(\Upsilon_K)}^2+\|\nabla \bl{w}\|_{L^2(\Upsilon_K)}^3\right).
\]

\qed

\subsubsection{Main estimates of the uniqueness result}
\q\ Subtracting the equation of $\bl{u}$ from the equation of $\tilde{\bl{u}}$, one finds
\be\label{SUBT}
\bl{w}\cdot\nabla \bl{w}+\bl{u}\cdot\nabla \bl{w}+\bl{w}\cdot\nabla \bl{u}+\nabla(\tilde{p}-p)-\Dl \bl{w}=0.
\ee
Multiplying $\bl{w}$ on both sides of \eqref{SUBT}, and integrating on $\mS_\zeta$, one derives
\be\label{Maint0}
-\int_{\mathcal{S}_\zeta}\bl{w}\cdot\Dl \bl{w}dx=-\int_{\mathcal{S}_\zeta}\bl{w}\big(\bl{w}\cdot\nabla \bl{w}+\bl{u}\cdot\nabla \bl{w}+\bl{w}\cdot\nabla \bl{u}+\nabla (\tilde{p}-p)\big)dx.
\ee
Using the divergence-free property and the Navier-slip boundary condition of $\bl{u}$ and $\tilde{\bl{u}}$, one deduces
\[
\begin{split}
-\int_{\mathcal{S}_\zeta}\bl{w}\cdot\Dl \bl{w}dx&=-\int_{\mS_\zeta}w_i\p_{x_j}(\p_{x_j}w_i+\p_{x_i}w_j)dx\\
&=\sum_{i,j=1}^2\int_{\mS_\zeta}\p_{x_j}w_i(\p_{x_j}w_i+\p_{x_i}w_j)dx-\sum_{i,j=1}^2\int_{\p\mS_\zeta}w_in_j(\p_{x_j}w_i+\p_{x_i}w_j)dx\\
&=2\int_{\mS_\zeta}|\mathbb{S} \bl{w}|^2dx+\al\int_{\p\mS_\zeta\cap\p\mS}|w_{\tau}|^2dS\\
&\hskip 1cm-\sum_{i=1}^2\int_{\{x_1=\zeta\}}w_i(\p_{x_1}w_i+\p_{x_i} w_1)dx_2+\sum_{i=1}^2\int_{\{y_1=-\zeta\}}w_i(\p_{x_1}w_i+\p_{x_i}w_1)dy_2.
\end{split}
\]
Here $\bl{n}=(n_1,n_2)$ is the unit outer normal vector on $\p\mS$. Then one concludes that
\[
-\int_{\mathcal{S}_\zeta}\bl{w}\cdot\Dl \bl{w}dx+\int_{\{x_1=\zeta\}}|\bl{w}||\nabla \bl{w}|dx_2+\int_{\{y_1=-\zeta\}}|\bl{w}||\nabla \bl{w}|dy_2\geq 2\int_{\mS_\zeta}|\mathbb{S} \bl{w}|^2dx+\al\int_{\p\mS_\zeta\cap\p\mS}|\bl{w}_{\mathrm{tan}}|^2dS.
\]
Then using the Korn inequality \eqref{KN1.1} in Lemma \ref{lem-korn}, we can achieve that
\be\label{Maint1}
\int_{\mS_\zeta}|\na \bl{w}|^2dx\leq C\lt(-\int_{\mathcal{S}_\zeta}\bl{w}\cdot\Dl \bl{w}dx+\int_{\{x_1=\zeta\}}|\bl{w}||\nabla \bl{w}|dx_2+\int_{\{y_1=-\zeta\}}|\bl{w}||\nabla \bl{w}|dy_2\rt).
\ee
Now we focus on the right hand side of \eqref{Maint0}. Applying integration by parts, one derives
\be\label{Maint3}
\begin{split}
-\int_{\mathcal{S}_\zeta}\bl{w}\big(\bl{w}\cdot\nabla \bl{w}+\nabla (\tilde{p}-p)\big)dx=&-\int_{\mS\cap\{x_1=\zeta\}}\bl{w}\cdot\bl{e_1}\left(\f{1}{2}|\bl{w}|^2+(\tilde{p}-p)\right)dx_2\\
&+\int_{\mS\cap\{y_1=-\zeta\}}\bl{w}\cdot\bl{e_1}'\left(\f{1}{2}|\bl{w}|^2+(\tilde{p}-p)\right)dy_2.
\end{split}
\ee
Applying H\"older's inequality, noting that $\bl{u}=\bl{v}+\bl{a}$, where $\bl{a}$ is the flux carrier constructed in Proposition \ref{Prop}, while $\bl{v}$ is the $H^1$-weak solution given in Section \ref{SEC3627}, one has
\be\label{Maint5}
\begin{split}
\left|-\int_{\mathcal{S}_\zeta}\big(\bl{w}\cdot\nabla \bl{u}\cdot \bl{w}+\bl{u}\cdot\na\bl{w}\cdot\bl{w}\big)dx\right|\leq&\, \|\nabla\bl{v}\|_{L^2(\mS_\zeta)}\|\bl{w}\|_{L^4(\mS_\zeta)}^2+\|\bl{v}\|_{L^4(\mS_\zeta)}\|\na\bl{w}\|_{L^2(\mS_\zeta)}\|\bl{w}\|_{L^4(\mS_\zeta)}\\
&+\|\na\bl{a}\|_{L^\i(\mS_\zeta)}\|\bl{w}\|_{L^2(\mS_\zeta)}^2+\|\bl{a}\|_{L^\i(\mS_\zeta)}\|\na\bl{w}\|_{L^2(\mS_\zeta)}\|\bl{w}\|_{L^2(\mS_\zeta)}\\
\leq&\, C\left(\|\bl{v}\|_{H^{1}(\mS_\zeta)}+\|\bl{a}\|_{W^{1,\i}(\mS_\zeta)}\right)\int_{\mS_\zeta}|\nabla \bl{w}|^2dx\\
\leq&\,\Phi e^{C \Phi}\int_{\mS_\zeta}|\nabla \bl{w}|^2dx.
\end{split}
\ee
Here in the second inequality, we have applied the Gagliardo-Nirenberg inequality and the Poincar\'e inequality \eqref{TORPIPEPOIN} in Lemma \ref{TORPOIN}, which indicate
\[
\|\bl{w}\|_{L^4(\mS_\zeta)}\leq C\left(\|\bl{w}\|^{1/2}_{L^2(\mS_\zeta)}\|\na\bl{w}\|^{1/2}_{L^2(\mS_\zeta)}+\|\bl{w}\|_{L^2(\mS_\zeta)}\right)\leq C\left(\int_{\mS_\zeta}|\nabla \bl{w}|^2dx\right)^{1/2}.
\]
Meanwhile, the third inequality in \eqref{Maint5} is guaranteed by \eqref{EOFV627} and Estimates for $\bl{a}$. Substituting \eqref{Maint1}, \eqref{Maint3} and \eqref{Maint5} in \eqref{Maint0}, one arrives
\[
\begin{split}
\int_{\mS_\zeta}|\nabla \bl{w}|^2dx\leq& C\left(\int_{\{x_1=\zeta\}}|\bl{w}|(|\na \bl{w}|+|\bl{w}|^2)dx_2+\int_{\{y_1=-\zeta\}}|\bl{w}|(|\na \bl{w}|+|\bl{w}|^2)dy_2+\Phi e^{C\Phi}\int_{\mS_\zeta}|\nabla \bl{w}|^2dx\right.\\
&\left.-\int_{\{x_1=\zeta\}}\bl{w}\cdot\bl{e_1}\left(\tilde{p}-p\right)dx_2+\int_{\{y_1=-\zeta\}}\bl{w}\cdot\bl{e_1}'\left(\tilde{p}-p\right)dy_2\right).
\end{split}
\]
Now one concludes that if $\Phi<<1$ being small enough such that
\bes
\Phi e^{C\Phi}<\f{1}{2},
\ees
then we achieve
\[
\begin{split}
\int_{\mS_\zeta}|\nabla \bl{w}|^2dx\leq C&\left(\int_{\{x_1=\zeta\}}|\bl{w}|(|\na \bl{w}|+|\bl{w}|^2)dx_2+\int_{\{y_1=-\zeta\}}|\bl{w}|(|\na \bl{w}|+|\bl{w}|^2)dy_2\right.\\
&\left.-\int_{\{x_1=\zeta\}}\bl{w}\cdot\bl{e_1}\left(\tilde{p}-p\right)dx_2+\int_{\{y_1=-\zeta\}}\bl{w}\cdot\bl{e_1}'\left(\tilde{p}-p\right)dy_2\right).
\end{split}
\]
Therefore,  one derives the following estimate by integrating with $\zeta$ on $[K-1,K]$, where $K\geq 2$:
\be\label{ET+0}
\begin{split}
\int_{K-1}^K\int_{\mathcal{S}_\zeta}|\nabla \bl{w}|^2dxd\zeta\leq C&\lt(\int_{\Upsilon_K^+}|\bl{w}|(|\na \bl{w}|+|\bl{w}|^2)dx+\int_{\Upsilon_K^-}|\bl{w}|(|\na \bl{w}|+|\bl{w}|^2)dy\right.\\
&\left.+\Big|\int_{\Upsilon_K^+}\bl{w}\cdot\bl{e_1}\left(\tilde{p}-p\right)dx\Big|+\Big|\int_{\Upsilon_K^-}\bl{w}\cdot\bl{e_1}'\left(\tilde{p}-p\right)dy\Big|\rt).
\end{split}
\ee
 Now we only handle integrations on $\Upsilon_K^+$ since the cases of $\Upsilon_K^-$ are similar. Using the Cauchy-Schwarz inequality and the Poincar\'e inequality Lemma \ref{POIN}, one has
\be\label{ET+1}
\int_{\Upsilon_K^+}|\bl{w}||\na \bl{w}|dx\leq\|\bl{w}\|_{L^2(\Upsilon_K^+)}\|\nabla \bl{w}\|_{L^2(\Upsilon_K^+)}\leq C\|\nabla \bl{w}\|^2_{L^2(\Upsilon_K^+)}.
\ee
Moreover, by H\"older's inequality and the Gagliardo-Nirenberg inequality, one writes
\[
\int_{\Upsilon_K^+}|\bl{w}|^3dx\leq C\left(\|\bl{w}\|_{L^2(\Upsilon_K^+)}^{2}\|\na \bl{w}\|_{L^2(\Upsilon_K^+)}+\|\bl{w}\|_{L^2(\Upsilon_K^+)}^{3}\right),
\]
which follows by the Poincar\'e inequality that
\[
\int_{\Upsilon_K^+}|\bl{w}|^3dx\leq C\|\na \bl{w}\|_{L^2(\Upsilon_K^+)}^{3}.
\]
Recalling Proposition \ref{P2.4}, one arrives at
\be\label{ET+2}
\left|\int_{\Upsilon_K^+}w_3\left(\tilde{p}-p\right)dx\right|\leq C\left(\|\bl{u}\|_{L^4(\Upsilon_K^+)}\|\nabla \bl{w}\|^2_{L^2(\Upsilon_K^+)}+\|\nabla \bl{w}\|_{L^2(\Upsilon_K^+)}^2+\|\nabla \bl{w}\|_{L^2(\Upsilon_K^+)}^3\right).
\ee
Substituting \eqref{ET+1}--\eqref{ET+2}, together with their related inequality on domain $\Upsilon_K^{-}$, in \eqref{ET+0}, one concludes
\be\label{FEST}
\int_{K-1}^K\int_{\mathcal{S}_\zeta}|\nabla \bl{w}|^2dxd\zeta\leq C\left(\|\nabla \bl{w}\|_{L^2(\Upsilon_K^+\cup\Upsilon_K^-)}^2+\|\nabla \bl{w}\|_{L^2(\Upsilon_K^+\cup\Upsilon_K^-)}^3\right).
\ee
\subsubsection{End of proof}

\q\ Finally, by defining
\[
Y(K):=\int_{K-1}^K\int_{\mathcal{S}_\zeta}|\nabla \bl{w}|^2dxd\zeta,
\]
\eqref{FEST} indicates
\[
Y(K)\leq C\left(Y'(K)+\left(Y'(K)\right)^{3/2}\right),\q\forall K\geq 1.
\]
By Lemma \ref{LEM2.3}, we derive
\[
\liminf_{\zeta\to\infty}K^{-3}Y(K)>0,
\]
that is, there exists $C_0>0$ such that
\[
\int_{K-1}^K\int_{\mathcal{S}_\zeta}|\nabla \bl{w}|^2dxd\zeta\geq C_0K^3.
\]
However, this leads to a paradox with the condition \eqref{sesti}. Thus $Y(K)\equiv0$ for all $K\geq 1$, which proves $\bl{u}\equiv \tilde{\bl{u}}$. This concludes the uniqueness.

\qed

\section{Asymptotic and regularity of the weak solution}\label{SECH}

\subsection{Decay estimate of the weak solution}

In this subsection we will show the weak solution constructed in the previous section decays exponentially to Poiseuille flows \eqref{PF} as $|x|\to\i$. Our proof is also valid for stationary Navier-Stokes problem on domains which is less regular, say an infinite pipe only with a $C^{1,1}$ boundary.

For the convenience of our further statement, we localize the problem in the following way: Denoting
\be\label{CUTFRAK}
\mS=\bigcup_{k\in\mathbb{Z}}\mathfrak{S}_k,
\ee
where
\[
\mathfrak{S}_k:=\left\{
\begin{array}{ll}
\mS\cap\left\{x\in\mathbb{R}^2:\,\left(\f{3k}{2}-1\right)Z_\Phi\leq x_1\leq\left(\f{3k}{2}+1\right)Z_\Phi\right\},
&\q k>0;\\[1.5mm]
\mS_{Z_\Phi},&\q k=0;\\[1.5mm]
\mS\cap\left\{x\in\mathbb{R}^2:\,\left(\f{3k}{2}-1\right)Z_\Phi\leq y_1\leq\left(\f{3k}{2}+1\right)Z_\Phi\right\},
&\q k<0,\\
\end{array}
\right.
\]
where $Z_{\Phi}=e^{2/\ve}\leq e^{C\Phi}$, while $\ve>0$ is a fixed small constant given in the construction of $\bl{a}$. Here is the main result of this subsection:


\begin{proposition}\label{PROP5.1}
Let the conditions of the item (ii) in Theorem \ref{PRO1.2} be satisfied and $(\bl{v},\,\Pi)$ is given in \eqref{EEPPSS} and \eqref{DEFPI}. Then there exist positive constants $C$, $\sigma$, depending only on $\Phi$, such that
\be\label{ASYP}
\begin{split}
\left\|\bl{u}-\boldsymbol{P}^L_{\Phi}\right\|_{H^1(\mS_L\backslash\mS_\zeta)}+\left\|\bl{u}-\boldsymbol{P}^R_{\Phi}\right\|_{H^1(\mS_R\backslash\mS_\zeta)}&\leq C\|\bl{v}\|_{H^1(\mS)}\exp(-\sigma\zeta),\\
\end{split}
\ee
for any $\zeta$ being large enough.
\end{proposition}

\qed

During the proof of Proposition \ref{PROP5.1}, we need the following refined estimate of the pressure field:

\begin{lemma}\label{RMKK54}
The reformulated pressure field $\Pi$ given in \eqref{DEFPI} enjoys the following uniform estimate:
\[
\sum_{k\in\mathbb{Z}}\|\Pi-\overline{\Pi}_{\mathfrak{S}_k}\|^2_{L^2(\mathfrak{S}_k)}\leq \Phi^2 e^{C\Phi}<\i.
\]
\end{lemma}
\pf Applying \eqref{EEEE0} in Lemma \ref{LEM312}, one deduces
\be\label{EEEE1}
\|\Pi-\overline{\Pi}_{\mathfrak{S}_k}\|_{L^2(\mathfrak{S}_k)}\leq C_k\|\na\Pi\|_{H^{-1}(\mathfrak{S}_k)}.
\ee
Notice that, each $\mathfrak{S}_k$ ($k\in\mathbb{Z}$) is congruent to an element in $\{\mathfrak{S}_{-1},\,\mathfrak{S}_0,\,\mathfrak{S}_1\}$. This indicates constants $C_k$ in estimates \eqref{EEEE1} above could be chosen uniformly with respect to $k\in\mathbb{Z}$. By equation
\[
\nabla\Pi=\text{div}\big(\nabla\bl{v}-\bl{v}\otimes\bl{v}-\bl{a}\otimes\bl{v}-\bl{v}\otimes\bl{a}\big)+\Dl \bl{a}+{C_{R}}{\eta(x_1)}\bl{e_1}+{C_{L}}{\eta(-y_1)}\bl{e_1}'-\bl{a}\cdot\nabla \bl{a},
\]
with both $\Dl \bl{a}+{C_{R}}{\eta(x_1)}\bl{e_1}+{C_{L}}{\eta(-y_1)}\bl{e_1}'$ and $\bl{a}\cdot\nabla \bl{a}$ vanish in $\mathfrak{S}_k$ with $|k|\geq 2$, one concludes from \eqref{EEEE1} that
\[
\begin{split}
\|\Pi-\overline{\Pi}_{\mathfrak{S}_k}\|_{L^2(\mathfrak{S}_k)}&\leq C\left(\|\na\bl{v}\|_{L^2(\mathfrak{S}_k)}+\|\bl{v}\|_{L^4(\mathfrak{S}_k)}^2+\Phi e^{C\Phi}\|\bl{v}\|_{L^2(\mathfrak{S}_k)}\right)+\Phi e^{C\Phi}\chi_{|k|\leq1}\\
&\leq C\|\bl{v}\|_{H^1(\mathfrak{S}_k)}\left(1+\Phi e^{C\Phi}+\|\bl{v}\|_{H^1(\mathfrak{S}_k)}\right)+\Phi e^{C\Phi}\chi_{|k|\leq1}.
\end{split}
\]
Here we have applied the Sobolev imbedding theorem and interpolations of $L^p$ spaces. This completes the proof of Lemma \ref{RMKK54}.

\qed

{\noindent\bf Proof of Proposition \ref{PROP5.1}: }
We only prove the estimate of term $\|\bl{u}-\boldsymbol{P}^R_{\Phi}\|_{H^1(\mS_R\backslash\mS_\zeta)}$ since the rest term is essentially identical. For $\zeta>Z_\Phi$, in $\mS_R\backslash\mS_\zeta$, the equation of $\bl{v}=\bl{u}-\bl{a}$ reads
\be\label{EW}
\bl{v}\cdot\nabla \bl{v}+\bl{a}\cdot\nabla\bl{v}+\bl{v}\cdot\na\bl{a}+\nabla\Pi-\Dl\bl{v}=0.
\ee
This is because
\[
\Dl \bl{a}+{C_{R}}{\eta(x_3)}\bl{e_1}+{C_{L}}{\eta(-y_1)}\bl{e_1}'-\bl{a}\cdot\nabla\bl{a}=\left(\Dl {P}^R_{\Phi}+{C_{L}}\right)\bl{e_1}=0,\q\text{in}\q\mS_R\backslash\mS_\zeta.
\]
In the following proof, we will drop (upper or lower) indexes ``$R$" for convenience. For any $Z_\Phi<\zeta\leq\zeta'<\zeta_1$, taking inner product with $\bl{v}$ on both sides of \eqref{EW} and integrating on $\mS_R\cap(\mS_{\zeta_1}\backslash\mS_{\zeta'})$, one has
\be\label{Maint000}
\un{\int_0^1\int_{\zeta'}^{\zeta_1}\bl{v}\cdot\Dl \bl{v}dx_1dx_2}_{LHS}=\un{\int_0^1\int_{\zeta'}^{\zeta_1}\big(\bl{v}\cdot\nabla \bl{v}+\bl{a}\cdot\nabla \bl{v}+\bl{v}\cdot\nabla \bl{a}+\nabla\Pi\big)\cdot\bl{v}dx_1dx_2}_{RHS}.
\ee
To handle the left hand side of \eqref{Maint000}, one first recalls the derivation of \eqref{Maint1} that
\[
\begin{split}
\int_0^1\int_{\zeta'}^{\zeta_1}\bl{v}\cdot\Dl\bl{v}dx_1dx_2=&-2\int_0^1\int_{\zeta'}^{\zeta_1}|\mathbb{S} \bl{v}|^2dx_1dx_2-\al\int_{\zeta'}^{\zeta_1}|\bl{v}_{tan}|^2\Big|_{x_2=1}dx_1-\al\int_{\zeta'}^{\zeta_1}|\bl{v}_{tan}|^2\Big|_{x_2=0}dx_1\\
&-\sum_{i=1}^2\int_0^1v_i(\p_{x_1}v_i+\p_{x_i}v_1)\Big|_{x_1=\zeta'}dx_2+\sum_{i=1}^2\int_0^1v_i(\p_{x_1}v_i+\p_{x_i}v_1)\Big|_{x_1=\zeta_1}dx_2.
\end{split}
\]
Applying Lemma \ref{lem-korn}, the Korn's inequality in a truncated stripe, one deduces the left hand side of \eqref{Maint000} satisfies
\be\label{LH}
\begin{split}
LHS\leq&C\left(-\int_0^1\int_{\zeta'}^{\zeta_1}|\na \bl{v}|^2dx_1dx_2+\int_0^1|\bl{v}||\na\bl{v}|\Big|_{x_1=\zeta'}dx_2+\int_0^1|\bl{v}||\na\bl{v}|\Big|_{x_1=\zeta_1}dx_2\right)
\end{split}
\ee
Using integration by parts for the right hand side of \eqref{Maint000}, one arrives
\be\label{RH}
\begin{split}
RHS=&\int_0^1\left(\f{1}{2}\left(v_1+P_\Phi\right)|\bl{v}|^2+v_1\Pi+P_\Phi(v_1)^2\right)\Big|_{x_1=\zeta_1}dx_2\\
&-\int_0^1\left(\f{1}{2}\left(v_1+P_\Phi\right)|\bl{v}|^2+v_1\Pi+P_\Phi(v_1)^2\right)\Big|_{x_1=\zeta'}dx_2\\
&-\int_0^1\int_{\zeta'}^{\zeta_1}\bl{v}\cdot\na \bl{v}\cdot \bl{a} dx_1dx_2.
\end{split}
\ee
Now we are ready to perform $\zeta_1\to\i$. To do this, one must be careful with the integrations on $\{x_1=\zeta_1\}\times(0,1)$ in both \eqref{LH} and \eqref{RH}. Recalling estimates of $(\bl{v},\Pi)$ in Theorem \ref{THM3.8} and Lemma \ref{RMKK54}, one derives
\be\label{WESTTT}
\|\bl{v}\|^2_{H^1(\mS)}+\|\bl{v}\|^4_{L^4(\mS)}+\sum_{k\in\mathbb{Z}}\|\Pi-\overline{\Pi}_{\mfS_k}\|^2_{L^2(\mfS_k)}\leq \Phi^2 e^{C\Phi}<\i.
\ee
Choosing $M:=\f{\Phi^2e^{C\Phi}}{Z_\Phi}$, one concludes that for any $k>1$, there exists a slice $\{x_1=\zeta_{1,k}\}\times(0,1)$ which satisfies
\[
\{x_1=\zeta_{1,k}\}\times(0,1)\subset\mS\cap\left\{x\in\mathbb{R}^2:\,\left(\f{3k}{2}-\f{1}{2}\right)Z_\Phi\leq x_1\leq\left(\f{3k}{2}+\f{1}{2}\right)Z_\Phi\right\}\subset\mfS_k,
\]
and it holds that
\[
\int_0^1\left(|\na\bl{v}|^2+|\bl{v}|^4+|\Pi-\overline{\Pi}_{\mfS_k}|^2\right)\Big|_{{x_1}=\zeta_{1,k}}dx_2\leq M.
\]
Otherwise, one has
\[
\|\bl{v}\|^2_{H^1(\mfS_k)}+\|\bl{v}\|^4_{L^4(\mfS_k)}+\|\Pi-\overline{\Pi}_{\mfS_k}\|^2_{L^2(\mfS_k)}{>Z_\Phi M=\Phi ^2e^{C\Phi}},
\]
which creates a paradox to \eqref{WESTTT}. Choosing $k_0>0$ being sufficiently large such that the sequence $\{\zeta_{1,k}\}_{k=k_0}^\i\subset [\zeta',\i)$, clearly one has $\zeta_{1,k}\nearrow\i$ as $k\to\i$. Moreover, using the trace theorem of functions in the Sobolev space $H^1$, one has
\[
\int_0^1|\bl{v}(x_1,x_2)|^2dx_2\leq C\int_{z>x_1}\int_0^1(|\bl{v}|^2+|\nabla \bl{v}|^2)(z,x_2)dx_2dz\to 0,\q\text{as}\q x_1\to\i.
\]
Noting that $\int_0^1v_1(\zeta_{1,k},x_2)dx_2=0$ for $k\geq k_0$, we deduce the following by the Poincar\'e inequality:
\[
\begin{split}
\left|\int_0^1v_3\Pi\Big|_{x_1=\zeta_{1,k}} dx_2\right|&=\left|\int_0^1v_3\left(\Pi-\overline{\Pi}_{\mfS_k}\right)\Big|_{x_1=\zeta_{1,k}} dx_2\right|\\
&\leq \left(\int_0^1|\bl{v}|^2\Big|_{x_1=\zeta_{1,k}}dx_2\right)^{1/2}\left(\int_0^1|\Pi-\overline{\Pi}_{\mfS_k}|^2\Big|_{x_1=\zeta_{1,k}}dx_2\right)^{1/2}\to 0,\q\text{as}\q k\to\i.
\end{split}
\]
Meanwhile, one finds
\[
\begin{split}
&\int_{\Sigma\times\{x_3=\zeta_{1,k}\}}|\bl{v}|\left(|\nabla\bl{v}|+|\bl{v}|^2\right)\Big|_{x_1=\zeta_{1,k}}dx_2\\
&\leq \left(\int_0^1\left(|\na\bl{v}|^2+|\bl{v}|^4\right)\Big|_{x_1=\zeta_{1,k}}dx_2\right)^{1/2}\left(\int_0^1|\bl{v}|^2\Big|_{x_1=\zeta_{1,k}}dx_2\right)^{1/2}\to 0,\q\text{as}\q k\to\i;
\end{split}
\]
and
\[
\int_0^1|P_\Phi||\bl{v}|^2\Big|_{x_1=\zeta_{1,k}}dx_2\leq\|P_\Phi\|_{L^\i(\mS_R)}\int_0^1|\bl{v}|^2\Big|_{x_1=\zeta_{1,k}}dx_2\to0,\q\text{as}\q k\to\i.
\]
Choosing $\zeta_1=\zeta_{1,k}$ ($k\geq k_0$) in \eqref{LH} and \eqref{RH}, respectively, and performing $k\to\i$, one can deduce that
\[
\begin{split}
\int_0^1\int_{\zeta'}^\i|\na \bl{v}|^2dx\leq&\,\un{C\int_0^1\int_{\zeta'}^\i\bl{v}\cdot\na \bl{v}\cdot \bl{a} dx}_{R_1}\\
&+C\int_0^1\Big(|\bl{v}|\left(|\bl{v}|^2+|P_\Phi||\bl{v}|+|\nabla\bl{v}|\right)+v_1\Pi\Big)\Big|_{x_1=\zeta'}dx_2.
\end{split}
\]
Using the Cauchy-Schwarz inequality, the Poincar\'e inequality in Lemma \ref{POIN}, and the construction of profile vector $\bl{a}$, one derives
\[
R_1\leq C\|P_\Phi\|_{L^\i(\mS_R)}\left(\int_0^1\int_{\zeta'}^\i|\na \bl{v}|^2dx\right)^{1/2}\left(\int_0^1\int_{\zeta'}^\i|\bl{v}|^2dx\right)^{1/2}\leq \f{C\al\Phi}{1+\al}\int_0^1\int_{\zeta'}^\i|\na \bl{v}|^2dx,
\]
which indicates the following estimate provided $\al\Phi$ is small enough such that $\f{C\al\Phi}{1+\al}<1$:
\be\label{MEST}
\int_0^1\int_{\zeta'}^\i|\na \bl{v}|^2dx\leq C\int_0^1\Big(|\bl{v}|\left(|\bl{v}|^2+|P_\Phi||\bl{v}|+|\nabla\bl{v}|\right)+v_3\Pi\Big)\Big|_{x_1=\zeta'}dx_2.
\ee
Denoting
\be\label{DEFG}
\mathcal{G}(\zeta'):=\int_0^1\int_{\zeta'}^\i|\na \bl{v}|^2dx,
\ee
and integrating \eqref{MEST} with $\zeta'$ on $(\zeta,\i)$, one arrives
\be\label{EEP0}
\int_\zeta^\i\mathcal{G}(\zeta')d\zeta'\leq C\left(\int_0^1\int_\zeta^\i\Big(|\bl{v}|\left(|\bl{v}|^2+|P_\Phi||\bl{v}|+|\nabla\bl{v}|\right)\Big)dx+\left|\int_0^1\int_\zeta^\i v_1\Pi dx\right|\right).
\ee
Applying the Poincar\'e inequality in Lemma \ref{POIN}, one deduces
\be\label{EEP1}
\int_0^1\int_\zeta^\i|\bl{v}|\left(|\bl{v}|^2+|P_\Phi||\bl{v}|+|\nabla\bl{v}|\right)dx\leq C\int_0^1\int_\zeta^\i|\nabla\bl{v}|^2dx.
\ee
Moreover, using a similar approach as in the proof of Proposition \ref{P2.4}, one notices that
\be\label{EEP2}
\begin{split}
\left|\int_0^1\int_\zeta^\i v_1\Pi dx\right|&\leq\sum_{m=1}^\i\left|\int_{\Upsilon_{\zeta+m}^+}v_1\Pi dx\right|\\
&\leq C\sum_{m=1}^\i\left(\|P_\Phi\|_{L^\infty(\Upsilon^{+}_{\zeta+m})}\|\nabla \bl{v}\|^2_{L^2(\Upsilon^{+}_{\zeta+m})}+\|\nabla \bl{v}\|_{L^2(\Upsilon^{+}_{\zeta+m})}^2+\|\nabla \bl{v}\|_{L^2(\Upsilon^{+}_{\zeta+m})}^3\right)\\
&\leq C\int_0^1\int_\zeta^\i|\nabla\bl{v}|^2dx.
\end{split}
\ee
Substituting \eqref{EEP1} and \eqref{EEP2} in \eqref{EEP0}, one arrives at
\[
\int_\zeta^\i\mathcal{G}(\zeta')d\zeta'\leq C\mathcal{G}(\zeta),\q\text{for any}\q\zeta>Z_\Phi.
\]
This implies
\[
\mathcal{N}(\zeta):=\int_\zeta^\i\mathcal{G}(\zeta')d\zeta'
\]
is well-defined for all $\zeta>Z_\Phi$, and
\be\label{EEP3}
\mathcal{N}(\zeta)\leq-C\mathcal{N}'(\zeta),\q\text{for any}\q\zeta>Z_\Phi.
\ee
Multiplying the factor $e^{C^{-1}\zeta}$ on both sides of \eqref{EEP3} and integrating on $[Z_\Phi,\zeta]$, one deduces
\[
\mathcal{N}(\zeta)\leq C\exp\left(-C^{-1}\zeta\right),\q\text{for any}\q\zeta>Z_\Phi.
\]
According to the definition \eqref{DEFG}, one has $\mathcal{G}$ is both non-negative and non-increasing. Thus
\[
\mathcal{G}(\zeta)\leq\int_{\zeta-1}^\zeta\mathcal{G}(\zeta')d\zeta'\leq\mathcal{N}(\zeta-1)\leq C\exp\left(-C^{-1}\zeta\right),\q\text{for any}\q\zeta>Z_\Phi+1.
\]
This completes the proof of the \eqref{ASYP} by choosing $\sigma=C^{-1}$.

\qed

\subsection{Higher-order regularity of weak solutions}
\subsubsection{$\bl{H^m}$-estimates of weak solutions}\label{SEC521}
Given an arbitrary $\phi\in C_c^\i(\overline{\mS}\,,\,\mathbb{R})$, with $\phi=0$ on $\p\mS$, direct calculation shows $\bl{\varphi}:=(-\p_{x_2}\phi\,,\,\p_{x_1}\phi)$ defines a well-defined test function in Definition \ref{defws}. By replacing $\bl{\varphi}$ with $(-\p_{x_2}\phi\,,\,\p_{x_1}\phi)$ in \eqref{EQU111}, and denoting $\o=\p_{x_2}v_1-\p_{x_1}v_2$, one deduces
\[
\begin{split}
-\int_{\mS}\o\Dl\phi dx+\int_{\p\mS}(\al-2\kappa)v_{\mathrm{tan}}\f{\p\phi}{\p\bl{n}}dS+\int_{\mS}\bl{v}\cdot\nabla \o\cdot \phi dx=&\int_{\mS}\big(\Dl b-\bl{a}\cdot\nabla b\big)\cdot\phi dx\\
&+\int_{\mS}\left(\bl{v}\cdot\na\bl{a}+\bl{a}\cdot\nabla\bl{v}\right)^{\perp}\cdot\na\phi dx,
\end{split}
\]
where $b=\p_{x_2}a_1-\p_{x_1}a_2$. This implies $\o$ solves the following linear elliptic problem weakly:
\be\label{PCURL}
\left\{
\begin{array}{ll}
-\Dl\o+\bl{v}\cdot\nabla\o=(\Dl b-\bl{a}\cdot\nabla b)-\na\cdot\left(\bl{v}\cdot\na\bl{a}+\bl{a}\cdot\nabla\bl{v}\right)^{\perp}, & \text{in}\q\mS;\\[2mm]
\o=\left(-2\kappa+\al\right)v_{\mathrm{tan}}, & \text{on}\q\p\mS.
\end{array}
\right.
\ee
Here $\bl{v}\in H^1(\mS)$ is treated as a known function solved in Section \ref{SEC3627}, while $\bl{a}$ is the smooth divergence-free flux carrier constructed in Section \ref{SEC32}. 

To study bounds of higher-order norms of the solution, we split the problem \eqref{PCURL} into a sequence of problems on bounded smooth domains. Recall the definition of $\mathfrak{S}_k$ in \eqref{CUTFRAK}, we denote the related cut-off function
\[
\psi_k=\left\{
\begin{array}{ll}
\psi\left(x_1-\f{3kZ_\Phi}{2}\right),&\q\text{for}\q k>0;\\[1mm]
\psi\left(y_1-\f{3kZ_\Phi}{2}\right),&\q\text{for}\q k<0,
\end{array}
\right.
\]
where $\psi$ is a smooth 1D cut-off function that satisfies:
\[
\left\{
\begin{array}{*{2}{ll}}
\mathrm{supp}\,\psi\subset\left[-9Z_\Phi/10\,,9Z_\Phi/10\right];\\
\psi\equiv 1,&\q\text{in }\left[-{4Z_\Phi}/{5}\,,{4Z_\Phi}/{5}\right];\\
0\leq\psi\leq1,&\q\text{in }\left[-Z_\Phi\,,Z_\Phi\right];\\
|\psi^{(m)}|\leq C/Z_\Phi^m\leq C,&\q\text{for }m=1,2. \\
\end{array}
\right.
\]
Meanwhile, $\psi_0$ is a 2D smooth cut-off function that enjoys
\[
\left\{
\begin{array}{*{2}{ll}}
\mathrm{supp}\,\psi_0\subset\mS_{9Z_\Phi/10};\\
\psi\equiv 1,&\q\text{in }\mS_{4Z_\Phi/5};\\
0\leq\psi\leq1,&\q\text{in }\mS_{Z_\Phi};\\
|\psi^{(m)}|\leq C/Z_\Phi^m\leq C,&\q\text{for }m=1,2. \\
\end{array}
\right.
\]
However, domains  $\mathfrak{S}_k$ given in previous subsection are only Lipschitzian, which may cause unnecessary difficulty in deriving higher-order regularity of $\o$. To this end, we introduce $\tilde{\mathfrak{S}}_k$, a bounded smooth domain which contains $\mathfrak{S}_k$, with its boundary $\p\tilde{\mathfrak{S}}_k\supset\p\mathfrak{S}_k\cap\p\mS$. In order to make the constants of specific inequalities (i.e. imbedding inequalities, trace inequalities, Biot-Savart law) on each $\tilde{\mathfrak{S}}_k$ ($k\in\mathbb{Z}$) being uniform, one chooses every $\tilde{\mathfrak{S}}_k$ with $k>0$ to be congruent to $\tilde{\mathfrak{S}}_1$, and every $\tilde{\mathfrak{S}}_k$ with $k<0$ to be congruent to $\tilde{\mathfrak{S}}_{-1}$. This can be guaranteed by the definition of $\mathfrak{S}_k$. By the splitting and constructions above, the ``distorted part" in the middle of the stripe is totally contained in $\mfS_0\subset\tilde{\mfS}_0$, and $\na\psi_k$ are totally supported away from this``distorted part" for each $k\in\mathbb{Z}$.

Multiplying \eqref{PCURL}$_1$ by $\psi_k$, we can convert the problem \eqref{PCURL} to related problem in domain $\tilde{\mfS}_k$, with $k\in\mathbb{Z}$:
\[
\left\{
\begin{array}{ll}
-\Dl\o_k+\bl{v}\cdot\nabla\o_k=\nabla\cdot\bl{F}_k+\bl{f}_k, & \text{in}\q\tilde{\mfS}_k;\\[2mm]
\o_k=g_k, & \text{on}\q\p\tilde{\mfS}_k.
\end{array}
\right.
\]
Here $\o_k=\psi_k\o$, while
\[
\begin{split}
\bl{F}_k=&-\psi_k\left(\bl{v}\cdot\nabla\bl{a}+\bl{a}\cdot\nabla\bl{v}\right)^{\perp}-2\o\nabla\psi_k;\\[2mm]
\bl{f}_k=&\,\,\psi_k\left(\Dl b-\bl{a}\cdot\nabla b\right)+\nabla\psi_k\cdot\left(\bl{v}\cdot\na \bl{a}+\bl{a}\cdot\na \bl{v}\right)^{\perp}+\o\left(\Dl\psi_k+\bl{v}\cdot\nabla\psi_k\right);\\[2mm]
g_k=&\left(-2\kappa+\al\right)\psi_k \bl{v}_{\mathrm{tan}}.
\end{split}
\]
Using Gagliardo-Nirenberg interpolation together with the trace theorem, it is not hard to derive
\be\label{EST0908}
\|\bl{F}_k\|_{L^2(\tilde{\mfS}_k)}+\|\bl{f}_k\|_{L^2(\tilde{\mfS}_k)}+\|g_k\|_{H^{1/2}(\p\tilde{\mfS}_k)}\leq C\|\bl{v}\|_{H^1(\mathfrak{S}_k)}\left(1+\Phi e^{C\Phi}+\|\bl{v}\|_{H^1(\mathfrak{S}_k)}\right),\q\forall k\in\mathbb{Z}.
\ee
Noting that the constant $C$ above is independent with $k$, due to congruent property of domains $\{\tilde{\mfS}_k\}_{k\in\mathbb{Z}}$. Therefore, using the classical theory of elliptic equations and \eqref{EST0908}, one derives
\[
\|\o_k\|_{H^1(\tilde{\mfS}_k)}\leq C\|\bl{v}\|_{H^1(\mathfrak{S}_k)}\left(1+\Phi e^{C\Phi}+\|\bl{v}\|_{H^1(\mathfrak{S}_k)}\right),\q\forall k\in\mathbb{Z}.
\]
Applying the Biot-Savart law, one derives
\[
\|\bl{v}\|_{H^2(\mathfrak{S}_k')}\leq C\left(\|\o_k\|_{H^1(\tilde{\mfS}_k)}+\|\bl{v}\|_{L^2(\mathfrak{S}_k')}\right)\leq C\|\bl{v}\|_{H^1(\mathfrak{S}_k)}\left(1+\Phi e^{C\Phi}+\|\bl{v}\|_{H^1(\mathfrak{S}_k)}\right),\q\forall k\in\mathbb{Z},
\]
where
\[
\mathfrak{S}_k'=\{x\in\mathfrak{S}_k:\,\psi_k=1\}.
\]
This implies, by summing over $k\in\mathbb{Z}$, that
\[
\begin{split}
\|\bl{v}\|^2_{H^2(\mS)}\leq& \sum_{k\in\mathbb{Z}}\|\bl{v}\|^2_{H^2(\mathfrak{S}_k')}\\
\leq&C\sum_{k\in\mathbb{Z}}\|\bl{v}\|^2_{H^1(\mathfrak{S}_k)}\left(1+\Phi e^{C\Phi}+\|\bl{v}\|_{H^1(\mathfrak{S}_k)}\right)^2\\
\leq&C\|\bl{v}\|^2_{H^1(\mS)}\left(1+\Phi^2 e^{C\Phi}+\|\bl{v}\|^2_{H^1(\mS)}\right).\\
\end{split}
\]
This concludes the global $H^2$-regularity estimate of $\bl{v}$. From this, similarly as we derive \eqref{EST0908}, one achieves a ``one-order upper" regularity of $\bl{F}_k$, $\bl{f}_k$ and $g_k$, for any $k\in\mathbb{Z}$, that is
\[
\|\bl{F}_k\|_{H^1(\tilde{\mfS}_k)}+\|\bl{f}_k\|_{H^1(\tilde{\mfS}_k)}+\|g_k\|_{H^{3/2}(\p\tilde{\mfS}_k)}\leq C_{\Phi},
\]
which indicates the $H^3$-regularity of $\bl{v}$. Following this bootstrapping argument, one deduces $\bl{v}$ is smooth and
\[
\|\bl{v}\|_{H^m(\mS)}\leq C_{\Phi,m},\q\forall m\in\mathbb{N}.
\]
This finished the proof of the regularity part of Theorem \ref{THMSA}.

\qed

\subsubsection{Exponential decay of higher-order norms}

Finally, the higher-order regularity and the $H^1$-exponential decay estimate in previous subsection, indicates the higher-order exponential decay. In fact, using Sobolev imbedding, we first need to show the following decay of the solution in $H^m$ norms, with $m\geq 2$:
\[
\begin{split}
&\|\bl{v}\|_{H^m(\mS_L\backslash\mS_\zeta))}+\|\bl{v}\|_{H^m(\mS_R\backslash\mS_\zeta)}\leq C_{\Phi,m}\left(\|\bl{v}\|_{H^1(\mS_L\backslash\mS_{\zeta-Z_\Phi})}+\|\bl{v}\|_{H^1(\mS_R\backslash\mS_{\zeta-Z_\Phi})}\right),\\[1mm]
\end{split}
\]
for all $\zeta>2Z_\Phi$. This is derived by using the method in the proof of Section \ref{SEC521}, but summing over $k\in\mathbb{Z}$ such that
\[
\mathrm{supp}\,\psi_k\cap\left(\mS\backslash\mS_\zeta\right)\neq\varnothing.
\]
Then, the proof is completed by the $H^1$ decay estimate \eqref{ASYP}. This finishes the proof of Theorem \ref{THMSA}.

\qed

\begin{remark}\label{RMMKK43}
For the pressure $p$, there exists two constants $C_{L},\,C_{R}>0$ (See \eqref{POSS1} and \eqref{POSS2}) , and a smooth cut-off function $\eta$ given in \eqref{ETA} such that: For any  $m\geq 0$,
\[
\left\|\na^m\na\left(p+{C_R\int_{-\i}^{x_1}\eta(s)ds}-{C_L\int_{-\i}^{-y_1}\eta(s)ds}\right) \right\|_{L^2(\mS)}\leq C_{\Phi,m}.
\]
Meanwhile, the following pointwise decay estimate holds: for all $|x|>>1$,
\[
\left|\na^m\na \left(p+{C_R\int_{-\i}^{x_1}\eta(s)ds}-{C_L\int_{-\i}^{-y_1}\eta(s)ds}\right)(x)\right|\leq C_{\Phi,m}\exp\left\{-\sigma_{\Phi,m}|x|\right\},
\]
where $C_{\Phi,m}$ and $\sigma_{\Phi,m}$ are positive constants depending on $\Phi$ and $m$. The subtracted term
\[
\pi_{\bl{P}}:=-{C_R\int_{-\i}^{x_1}\eta(s)ds}+{C_L\int_{-\i}^{-y_1}\eta(s)ds}
\]
is set to balance the pressure of the Poiseuille flows.
\end{remark}

\qed

\section*{Data availability statement}
\addcontentsline{toc}{section}{Data availability statement}

\q\ Data sharing is not applicable to this article as no datasets were generated or analysed during the current study.

\section*{Conflict of interest statement}
\addcontentsline{toc}{section}{Conflict of interest statement}
\q\ The authors declare that they have no conflict of interest.

\section*{Acknowledgments}
\addcontentsline{toc}{section}{Acknowledgments}
\q\ Z. Li is supported by Natural Science Foundation of Jiangsu Province (No. BK20200803) and National Natural Science Foundation of China (No. 12001285). X. Pan is supported by National Natural Science Foundation of China (No. 11801268, 12031006). J. Yang is supported by National Natural Science Foundation of China (No. 12001429).

\medskip
\medskip

{\footnotesize

{\sc Z. Li: School of Mathematics and Statistics, Nanjing University of Information Science and Technology, Nanjing 210044, China}

  {\it E-mail address:}  zijinli@nuist.edu.cn

\medskip

 {\sc X. Pan: College of Mathematics and Key Laboratory of MIIT, Nanjing University of Aeronautics and Astronautics, Nanjing 211106, China}

  {\it E-mail address:}  xinghong\_87@nuaa.edu.cn

\medskip

 {\sc J. Yang: School of Mathematics and Statistics, Northwestern Polytechnical University, Xi'an 710129, China}

  {\it E-mail address:} yjqmath@nwpu.edu.cn
}
\end{document}